\documentclass[10pt]{article}
\usepackage{amsmath,amssymb,latexsym,epsfig,subfigure,multirow}
\usepackage{amsthm}
\usepackage{verbatim}
\usepackage{fancyhdr}
\newtheorem{theorem}{Theorem}[section]
\newtheorem{lemma}{Lemma}[section]
\newtheorem{definition}{Definition}[section]

\newtheorem{remark}{Remark}

\newcommand{\pa}{\partial}
\newcommand{\f}{\frac}
\newcommand{\abs}[1]{\lvert#1\rvert}
\begin{document}


\title{Matched Interface and Boundary Method for Elasticity Interface Problems }

\author{Bao Wang$^{1}$, Kelin Xia$^{1}$ and Guo-Wei Wei$^{1,2}$
\footnote{
Corresponding author. Email: wei@math.msu.edu}  \\
\\
\small \it     $^1$Department of Mathematics, Michigan State University, East Lansing, MI 48824, USA \\
\small \it     $^2$Department of Electrical and Computer Engineering, \\
\small \it          Michigan State University, East Lansing, MI 48824, USA  }

\date{\today}

\maketitle

\begin{abstract}
Elasticity theory is an important component of continuum mechanics and has had widely spread  applications in science and engineering. Material interfaces are ubiquity in nature and man-made devices, and often give rise to discontinuous coefficients in the governing elasticity equations. In this work, the matched interface and boundary (MIB) method is developed to address  elasticity interface problems. Linear elasticity theory for both isotropic homogeneous and inhomogeneous media is employed.  In our approach, Lam$\acute{e}$'s parameters can have jumps across the interface  and are allowed to be position dependent in modeling isotropic inhomogeneous material. Both strong discontinuity, i.e., discontinuous solution, and weak discontinuity, namely, discontinuous derivatives of the solution, are considered in the present study. In the proposed method, fictitious values  are utilized so that the standard central finite different schemes can be employed  regardless of the interface. Interface jump conditions are enforced on the interface, which in turn, accurately determines fictitious values.  We design new MIB schemes to account for complex interface geometries.  In particular, the cross derivatives in the elasticity equations are difficult to handle for complex interface geometries. We propose secondary fictitious values and construct  geometry based interpolation schemes to overcome this difficulty. Numerous analytical examples are used to validate the accuracy, convergence and robustness of the present  MIB method for elasticity interface problems with both small and large curvatures, strong and weak discontinuities, and constant and variable coefficients. Numerical tests indicate second order accuracy in both $L_\infty$ and $L_2$ norms.
\end{abstract}



\section{Introduction}
Elasticity interface problems play significant roles in continuum mechanics in which elasticity theory and related governing partial differential equations (PDEs) are commonly employed to describe various material behaviors.  For this class of problems, an interface description in the elasticity theory is  indispensable whenever there are voids, pores, inclusions, dislocations, cracks or composite structures in materials \cite{Dvorak:2013,Fries:2010,Sukumar:2001,Stolarska:2001}. Elasticity interface problems are particularly important in tissue engineering, biomedical science and biophysics \cite{Wei:2009,Wei:2013,KLXia:2013d}.  In many situations, the interface is not static such as  fluid-structure interfacial boundaries \cite{WangXS:2009}. Discontinuities in material properties often occur over the interface. Mathematically, there are  two types of discontinuities, namely, strong discontinuities and weak discontinuities. Strong discontinuities are referred to situations where the displacement has jumps across the interface. In contrast, weak discontinuities are concerned  with jumps in the gradient of the displacement, whereas the displacement is still continuous. In the linear elasticity, the stress-strain relation is governed by the constitutive equation. For isotropic homogeneous material, constitutive equations can be determined with any two terms of  bulk modulus, Young's modulus, Lam$\acute{e}$'s first parameter, shear modulus, Poisson's ratio, and P-wave modulus \cite{Anandarajah:2010}. If these moduli are position dependent functions, related constitutive equations can be used to describe elasticity property of isotropic inhomogeneous media. In seismic wave equations, inhomogeneity is accounted by assuming Lam$\acute{e}$'s parameters to be a position dependent function \cite{Shearer:1999}. This model is also used in the elasticity   analysis of  biomolecules \cite{Wei:2009,Wei:2013,KLXia:2013d}.

The study of the analytical solution for elasticity interface problems dated back to Eshelby in 1950s \cite{Eshelby:1956,Eshelby:1957}.  Working on inclusion and inhomogeneity problems, Eshelby found that for an infinite and elastically isotropic system with an ellipsoidal inhomogeneity, the eigenstrain distribution is uniform inside the inhomogeneity when it is subjected to a uniformly applied stress \cite{Eshelby:1956,Eshelby:1957}. Much progress has been made on this area in the past few decades. Recently, semianalytic approaches for finding stress tensors have been proposed for arbitrarily shaped inhomogeneity \cite{Mathiesen:2008}.

Computationally, elasticity interface problems are more difficult than the corresponding  Poisson  interface problems  because of the vector equation and cross derivatives. However, many numerical methods have been designed for elasticity interface problems. Based on meshes used, these methods can be classified  into two types, i.e., algorithms relied on body-fitting meshes and algorithms based on special interface schemes. For the first type, meshes are generated to fit to the geometry of the interface without  cutting through the interface. Therefore, adaptive meshes with local refinement techniques are frequently employed \cite{XuZL:2003}. In the second type of algorithms, meshes are allowed to cut through the interface and particular schemes are designed to incorporate the interface information into the element shape function or discretization scheme. Immersed interface method (IIM) \cite{LeVeque:1994} has been used to solve  elasticity interface problems for isotropic homogenous media \cite{YangXZ:2003,YGong:2010}. In this finite difference based algorithm, a local optimization scheme is designed for irregular grid points and the finial linear equation with a non-symmetric matrix is solved by special solvers like BICG and GRMES. Second order accuracy is obtained \cite{YangXZ:2003}.  A second-order sharp numerical method  has been developed for linear elasticity equations \cite{Theillard:2013}. Finite element based methods are also proposed for elasticity interface problems. Among them, the partition of unity method (PUM), the generalized finite element method (GFEM) and extended finite element method (XFEM) are developed to capture the non-smooth property of the solution over the interface   by adding  enrichment functions to the approximation \cite{Sukumar:2001,Stolarska:2001,Fries:2010}. Through the weak enforcement of the continuity, discontinuous Galerkin based methods have been employed to simulate strong and weak discontinuities \cite{Hansbo:2002,Becker:2009,Mergheim:2006}. Recently,  immerse finite element (IFM) method has been proposed to solve elasticity problems with inhomogeneous jump conditions \cite{LiZL:2005,XieH:2011,YChang:2012}. In this approach, finite element basis functions are adjusted locally to satisfy the jump conditions across the interface. Sharp-edged interface is considered for a special elasticity interface problem \cite{HouSM:2012}.  Lin, Sheen and Zhang have proposed a bilinear IFM and further modified it to a locking-free version \cite{LinT:2012,LinT:2013}.  For both compressible and nearly incompressible media, this method works well and offers second order accuracy. Recently, immersed meshfree Galerkin method has also been proposed for composite solids \cite{CTWu:2013}. Most recently, a Nitsche type method has been proposed for elasticity interface problems \cite{Michaeli:2013}.  Given the importance of elasticity interface problems in science and engineering, it is expected that more efficient numerical methods will be developed for this class of problems in the near future.

The matched interface and boundary (MIB) method was originally developed for solving Maxwell's equations \cite{Zhao:2004} and elliptic interface problems  \cite{Yu:2007,Yu:2007a,Zhou:2006c,Zhou:2006d,Geng:2007a}. A unique feature of the MIB method is that it provides a systematic procedure to achieve arbitrarily high order convergence for simple interfaces  \cite{Zhao:2004,Zhou:2006c} and second order accuracy for arbitrarily complex interface geometry \cite{Yu:2007,Yu:2007a}. The essential idea is to introduce fictitious values at irregular mesh points which form fictitious domains \cite{JCPWei:1999} so that  standard finite difference schemes can   still be used across the interface.  The lowest order interface jump conditions are iteratively enforced at the interface which determines fictitious values on fictitious domains.   Typically, whenever possible, a high-dimensional  interface problem is split into  simple one-dimensional (1D) interface problems, similar to our earlier discrete singular convolution algorithm \cite{JCPWei:1999}. Due to the great flexibility in the construction of   fictitious approximations, the MIB method has been shown to  deliver up to 16th order accuracy for simple interfaces \cite{Zhao:2004,Zhou:2006c} and  robust second order accuracy for arbitrarily complex interface geometry with geometric singularities (i.e., non-smooth interfaces with Lipschitz continuity) \cite{Yu:2007,Yu:2007a} and singular sources \cite{Geng:2007a}.  In the past decades, MIB method has been applied to a variety of problems. In computational biophysics, an MIB based Poisson-Boltzmann solver, MIBPB \cite{DuanChen:2011a}, has been constructed   for the analysis of the electrostatic potential of biomolecules \cite{Yu:2007, Geng:2007a,Zhou:2008b}, molecular dynamics \cite{Geng:2011} and charge transport phenomenon \cite{QZheng:2011a, QZheng:2011b}.  Zhao has developed robust MIB schemes for the Helmholtz problems \cite{SZhao:2010a,SZhao:2008a}. A second order accurate MIB method is constructed by Zhou and coworkers to solve the Navier-Stokes equations with discontinuous viscosity and density \cite{YCZhou:2012a}. Recently, the MIB method has been used to solve elliptic equations with multi-material interfaces \cite{KLXia:2011}.


The objective of the present paper is to introduce the MIB method for solving  elasticity interface problems. We consider both strong and weak discontinuities for isotropic homogeneous and inhomogeneous media. Computationally, the cross derivative terms in the elasticity model  give rise to  a new  challenge for the MIB method when the interface geometry is  complex. To overcome this difficulty, we modify the tradition fictitious definition and redefine   fictitious values. With the MIB dimension splitting technique, a new fictitious representation is generated for each irregular mesh point based on   elastic jump conditions and local geometry. Secondary fictitious values are constructed by the interpolation of these fictitious values and function values. We have designed schemes to deal with both small curvature and large curvature for complex interface geometries. To validate our method, analytical tests for different types of discontinuities and interface geometries are constructed. We demonstrate the second order accuracy of our MIB schemes for elasticity interface problems.

The rest of this paper is organized as follows.  The basic setting of  elasticity interface problems is presented in Section \ref{theory}. The linear elasticity equations, interface jump conditions and constitutive laws are discussed in detail to facilitate further consideration.
 Section \ref{algorithm} is devoted to the construction of MIB algorithms. General fictitious schemes are proposed for elasticity interface problems. Secondary fictitious values are introduced for cross derivative terms. Our method is extensively validated by   analytical tests with complex interface geometries in Section \ref{validation}.  This paper ends with a conclusion.


\section{Formulation of the elasticity interface problem}\label{theory}

In this section, the elasticity problem with material interfaces is formulated. First, the governing equations of the linear elasticity interface problem are derived. Then the weak solution to the governing equation is defined.  Based on the weak solution, the interface conditions are derived for the linear elasticity interface problem. Finally, the Dirichlet boundary condition is employed to make the linear elasticity interface problem    computationally well-posed.

\subsection{Linear elasticity equations}

When solid objects are subjected to external or internal loads, they deformed and lead to stress. If the deformation of the solid is relatively small,   linear relationships between the components of stress and strain are maintained. Consequently, linear elasticity theory is valid. In practice, linear elasticity theory is applicable to a wide range of  natural and engineering materials, and thus  extensively used in structural analysis and engineering design.

Figure \ref{illus_elas} illustrates the displacement in a two-dimensional (2D) elastic motion.  The displacement under an infinitesimal perturbation in position $\delta {\bf x}$ can be approximated by the linear term
\begin{equation}
\label{linear_app}
u_i(\mathbf{x+\delta x})\approx u_i(\mathbf{x})+\sum^2_{j=1} \frac{\partial u_i(\mathbf{x})}{\partial x_j} \delta x_j :=u_i(\mathbf{x})+\delta u_i,\ \  i=1, 2,
\end{equation}
where $\mathbf{x}=(x_1, x_2)$ is the position of a point of the un-deformed elastic body, $u_i$ is the $i$th component of the displacement vector and $\delta u_i$  is the relative displacement
\begin{equation}
\label{dis_app}
\delta u_i=\sum^2_{j=1} \frac{\partial u_i(\mathbf{x})}{\partial x_j} \delta x_j, \ \ i=1, 2.
\end{equation}

\begin{figure}[!ht]
\small
\centering
\includegraphics[width=6cm,height=5cm]{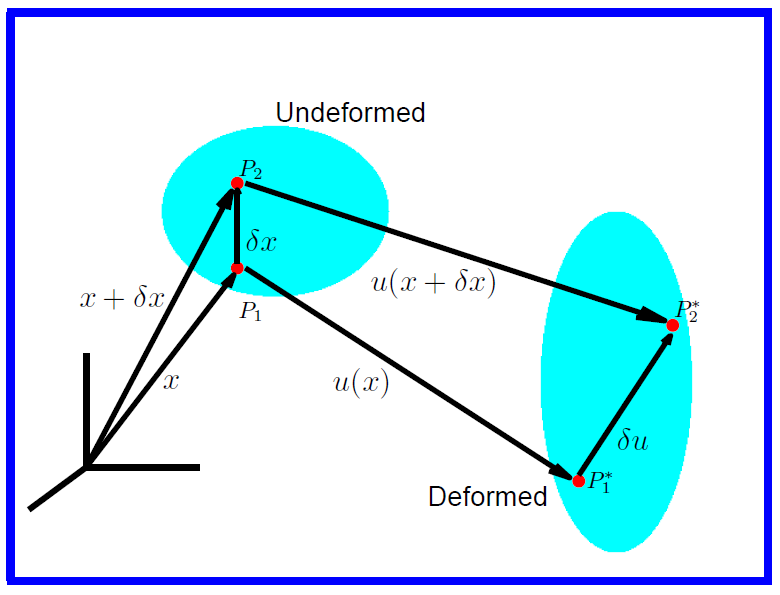}
\caption{Deformation in the continuum where two points $P_1$ and $P_2$ are deformed to $P_1^*$ and $P_2^*$ respectively.}
\label{illus_elas}
\end{figure}

The Cauchy's infinitesimal strain tensor, or the strain tensor $\sigma$ for simplicity, is defined as
\begin{equation}
\label{CST}
\sigma_{ij}=\frac{1}{2}\left(\frac{\partial u_j}{\partial x_i}+\frac{\partial u_i}{\partial x_j}\right), \ \ i, j=1, 2,
\end{equation}
where $\sigma_{ij}$ is the $ij-$th element of  $\sigma$.
In a compact notation,  the linear strain tensor is given by
\begin{equation}
\sigma=\frac{1}{2}\left(\nabla \mathbf{u}+(\nabla \mathbf{u})^T\right),
\end{equation}
where $\mathbf{u}=\left(u_1(x_1, x_2), u_2(x_1, x_2)\right)^T$ is the position of the point $(x_1, x_2)$ of the deformed elastic body.
Obviously, the strain tensor describes the total displacement. Physically, the Hooke's law states that the strain must lead to stress. Mathematically,  the constitutive equation between strain and stress tensors is given by
\begin{equation}
\label{const_eq}
T_{ij}=c_{ijkl}\sigma_{kl},
\end{equation}
where $T_{ij}$ is the $ij-$th element of  stress tensor $\mathbb{T}$ and $c_{ijkl}$ is  the $ijkl-$th element of   elastic moduli or  stiffness tensor $\mathbb{C}$,  which is a fourth  order tensor describing properties of the material. In the constitutive equation, the stress tensor is expressed through the contraction between the strain tensor and the stiffness tensor.

For isotropic homogeneous media, there is no  preferred direction in the stiffness tensor. Therefore, the stress-strain relation can be dramatically simplified.  We utilize  Lam\'{e}'s  parameter $\lambda$ and  shear modulus  $\mu$ to simplify the constitutive equation as
\begin{equation}
\label{con_2}
T_{ij}=\lambda {\rm tr}(\sigma)\delta_{ij}+2\mu \sigma_{ij},
\end{equation}
where $\delta_{ij}$ is the kronecker function, ${\rm tr}(\sigma)$ is the trace of the strain tensor. In a compact notation,  the stress tensor is
\begin{equation}
\mathbb{T}=\lambda {\rm tr}(\sigma)I+2\mu \sigma,
\end{equation}
where $I$ is the identity tensor.

In practical applications, one is often interested in the description of elasticity motion. By the Newton's second law, the motion of elasticity body is governed by
\begin{equation}\label{GEq1}
\nabla\cdot \mathbb{T}+\mathbf{F}=\frac{\partial^2\mathbf{u}}{\partial t^2},
\end{equation}
where $\mathbf{F}=(F_1, F_2)^T$ is the external force on the elastic body.  This equation can be more rigorously derived from the variation principle \cite{Wei:2009}.
The static state of the elastic motion is then governed by:
\begin{equation}\label{elas_inhom}
\nabla\cdot \mathbb{T}+\mathbf{F}=\mathbf{0}.
\end{equation}
In many applications, Eq. (\ref{elas_inhom}) is solved to obtain the deformation under a given force.

\subsection{Interface jump conditions}

Consider the static state of two-phase elastic body motions in domain $\Omega\subset \mathbb{R}^2$. Let us suppose that the two-phase elastic motions  are separated by interface $\Gamma$, which separates the whole domain $\Omega$ into two sub-domains $\Omega^+$ and $\Omega^-$,
i.e., $\Omega=\Omega^+ \cup \Omega^- \cup \Gamma$, as shown in Figure \ref{inter_pro}.

\begin{figure}[!ht]
\small
\centering
\includegraphics[width=5cm,height=5cm]{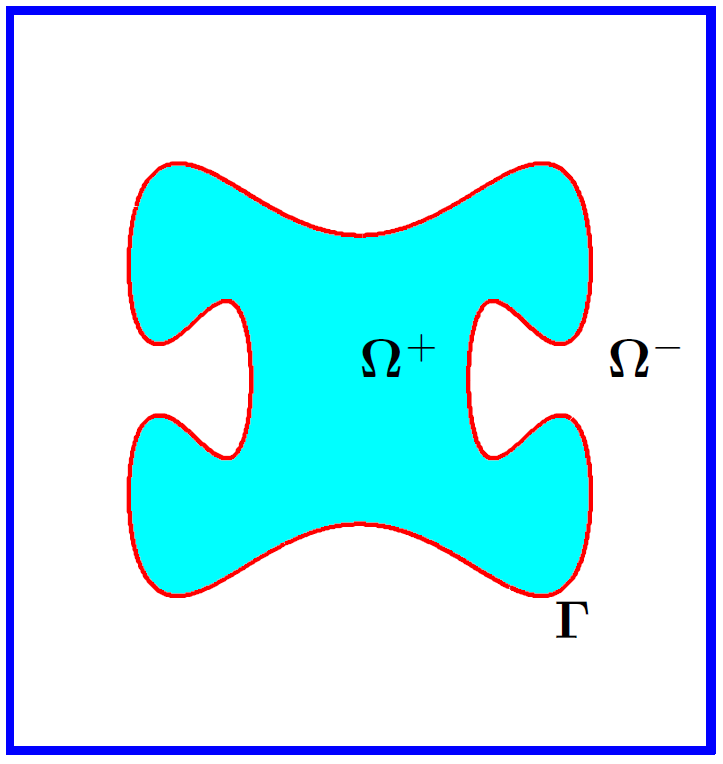}
\caption{Illustration of a  two-phase elasticity interface problem. The interface $\Gamma$ separates the whole computational domain into two parts $\Omega^+$ and $\Omega^-$.}
\label{inter_pro}
\end{figure}

\subsubsection{Weak solution of homogeneous elasticity equations}

In this section, we define the weak solution to   homogeneous linear elasticity equations without the body force
\begin{equation}
\label{hom_elas}
\nabla\cdot \mathbb{T}=\mathbf{0}.
\end{equation}
For simplicity, we   temporally consider the first equation in homogeneous elasticity equations (\ref{hom_elas}), i.e.,
$$
\nabla\cdot \mathbb{T}(:, 1)=0,
$$
where $\mathbb{T}(:, 1)$ denotes the first column of the matrix of the stress tensor $\mathbb{T}$, and let $\mathbb{\tilde{T}}:=\mathbb{T}(:, 1)$,
thus the first component of the homogeneous elasticity equation can be written as
\begin{equation}
\label{1_com_hom}
\nabla\cdot \mathbb{\tilde{T}}=0.
\end{equation}

To define the weak solution to the homogeneous elasticity equations, we choose the test function space to be $C^\infty_0(\Omega)$.
Multiplying Eq. (\ref{1_com_hom}) by $\phi\in C^\infty_0(\Omega)$ and integrating the obtained equation over the whole domain
$\Omega$ yield
$$
\int_\Omega \phi\nabla \cdot\tilde{\mathbb{T}}dV=0.
$$
Applying the integration by parts to the above equation, and note that $\phi$ is of compact support, one has
\begin{equation}
\label{weak1}
\int_\Omega \nabla\phi\cdot \tilde{\mathbb{T}}dV=0.
\end{equation}
Similarly,  a form  similar to Eq. (\ref{weak1})  can be derived for the other equation in Eq. (\ref{hom_elas}).

\begin{definition}
\textbf{Weak Solution}
$\tilde{\mathbb{T}}$ is said to be the weak solution of the equation $\nabla\cdot \tilde{\mathbb{T}}=0$ provided equation (\ref{weak1}) holds for
all $\phi\in C^\infty_0(\Omega)$.
\end{definition}

\subsubsection{Interface jump conditions}
In this part, interface jump conditions are formulated for  linear elasticity equation  (\ref{elas_inhom}).

\begin{lemma}
For 2D inhomogeneous linear elasticity equation (\ref{elas_inhom}), if the force term $\mathbf{F}$ has a potential function $U$, i.e.,
$\nabla U=\mathbf{F}$, then it can be expressed in a conservative form, i.e., homogeneous form.
\end{lemma}

\begin{proof}
Without loss of generality, we only prove the statement for the first equation in    equation (\ref{elas_inhom}), which is:
\begin{equation}
\label{inhom_2}
\nabla \cdot \mathbb{T}(:, 1)+U_x=0,
\end{equation}
where $\mathbb{T}(:, 1)$ is the first column of the tonsorial  matrix $\mathbb{T}$.

For convenient, we denote $\mathbb{T}(:, 1):=\left(T_{11}, T_{21}\right)^T$, then   equation (\ref{inhom_2})
becomes:
$$
\frac{\partial (T_{11}+U)}{\partial x}+\frac{\partial T_{21}}{\partial y}=0,
$$
thus, there exists another 2 by 2 tensor $\hat{\mathbb{T}}$, where $\hat{\mathbb{T}}-\mathbb{T}$ only depends on potential function $U$, such that
\begin{equation}
\label{inhom_3}
\nabla\cdot \hat{\mathbb{T}}=0.
\end{equation}
Furthermore, Eq. (\ref{elas_inhom}) is equivalent to Eq. (\ref{inhom_3}).
\end{proof}

\begin{theorem}
For  2D linear elasticity equations (\ref{elas_inhom}), if the source term $\mathbf{F}$ has a potential function $U$, i.e.,
$\nabla U=\mathbf{F}$, then across the interface, the weak solution should satisfy following interface conditions
\begin{equation}
\label{normal_inter}
[\mathbb{T}\cdot \mathbf{n}]=\mathbf{T},
\end{equation}
where $\mathbf{T}$ is a vector-valued function, $[*]$ is the difference of quantity ``*" across the interface and
$\mathbf{n}$ is the normal direction of the interface.
\end{theorem}

\begin{proof}
By the above lemma, there exists another second order tensor $\hat{\mathbb{T}}$, such that $\nabla \cdot \hat{\mathbb{T}}=\mathbf{0}$.
Without loss of generality, we only prove the interface condition for the first equation in  equations $\nabla \cdot \hat{\mathbb{T}}=\mathbf{0},$ and denote it as $\nabla \cdot \hat{\mathbb{T}}_1=0.$

Since $\Omega=\Omega^+ \cup \Gamma \cup \Omega^-$, and the interface is of measure zero, hence $\forall \phi \in C_0^\infty(\Omega)$, the following equation holds
\begin{equation}
\label{inter}
0=\int_\Omega \nabla\phi\cdot \hat{\mathbb{T}}_1 dV=\int_{\Omega^+} \nabla\phi\cdot \hat{\mathbb{T}}_1 dV+\int_{\Omega^-} \nabla\phi\cdot \hat{\mathbb{T}}_1 dV.
\end{equation}
Integrating by part applied to the first term in the right hand side of Eq. (\ref{inter}) yields
\begin{eqnarray*}
\int_{\Omega^+}\nabla\phi\cdot\hat{\mathbb{T}}_1dV
&=&\int_{\Omega^+}\nabla\cdot(\phi\hat{\mathbb{T}}_1)dV-\int_{\Omega^+}\phi \nabla\cdot \hat{\mathbb{T}}_1dV\\
&=&\int_{\Gamma}\phi\hat{\mathbb{T}}_1\cdot \mathbf{n}dS-\int_{\Omega^+}\phi \nabla\cdot \hat{\mathbb{T}}_1dV.
\end{eqnarray*}

In  region $\Omega^+$, note $\nabla \cdot \hat{\mathbb{T}}_1=0$, therefore:
\begin{equation}
\int_{\Omega^+}\nabla\phi\cdot\hat{\mathbb{T}}_1dV=\int_{\Gamma}\phi\hat{\mathbb{T}}_1^+\cdot \mathbf{n}dS,
\end{equation}
where $\hat{\mathbb{T}}_1^+$ means that the value was evaluated by taking the limit from the region $\Omega^+$.

Similarly, we have
$$
\int_{\Omega^-}\nabla\mathbf{\phi}\cdot \hat{\mathbb{T}}_1dV=-\int_{\Gamma}\mathbf{\phi}\hat{\mathbb{T}}^-_1\cdot \mathbf{n}dS,
$$
where the minus sign occurs because the outward normal $\mathbf{n}$ for $\Omega^+$ is the inward normal for $\Omega^-$.
Therefore, we have
$$
\int_{\Gamma}\phi\hat{\mathbb{T}}_1^+\cdot \mathbf{n}dS=\int_{\Gamma}\phi\hat{\mathbb{T}}_1^-\cdot \mathbf{n}dS,
$$
i.e.,
$$
\int_{\Gamma}\phi(\hat{\mathbb{T}}_1^+-\hat{\mathbb{T}}_1^-)\cdot \mathbf{n}dS=0.
$$
Since test function $\phi$ is arbitrary, we have $\hat{[\mathbb{T}}_1\cdot \mathbf{n}]=0$. In the same manner we can show that
$[\hat{\mathbb{T}}_2\cdot \mathbf{n}]=0$, where $\hat{\mathbb{T}}_2=\mathbb{T}(:,2)$. Therefore we have  $\nabla \cdot \hat{\mathbb{T}}=\mathbf{0}$.
According to the above lemma, we also have:
$$
[\mathbb{T}\cdot \mathbf{n}]=[(\mathbb{\hat{T}}+\mathbb{T}-\mathbb{\hat{T}})\cdot \mathbf{n}]=[(\mathbb{T}-\mathbb{\hat{T}})\cdot \mathbf{n}]:=\mathbf{T}.
$$
\end{proof}

Similarly, we can prove the following result.
\begin{theorem}
 Let $\mathbb{T}$ be a $2$nd order tensor in $ \mathbb{R}^n$, for  elasticity equations
\begin{equation}
\nabla \cdot \mathbb{T}+\mathbf{F}=\mathbf{0},
\end{equation}
where $\mathbf{F}$ is an $n$-dimensional vector-valued function, and $\mathbf{0}\in \mathbb{R}^n$.

If the force term has a potential function $U$, i.e., $\nabla U=\mathbf{F}$, then across the interface, the weak solution should satisfy the
following interface conditions
\begin{equation}
\label{inter2}
[\mathbb{T}\cdot \mathbf{n}]=\mathbf{T},
\end{equation}
where $\mathbf{T}$ is an $n$-dimensional vector-valued function and  $\mathbf{n}$ is the normal direction of the interface.
\end{theorem}

\begin{remark}
In interface condition $[\mathbb{T}\cdot \mathbf{n}]=\mathbf{T}$,  $\mathbf{T}$ is a given vector-valued function on the interface
$\Gamma$, which measures the jump of the traction $\mathbb{T}\cdot\mathbf{n}$ across the interface $\Gamma$.
\end{remark}

Moreover, physically,  we usually enforce a jump conditions to ensure that the material has no fracture in the weak discontinuity setting, i.e., it is continuous across the interface
$$
[\mathbf{u}]=\mathbf{0}.
$$
However, fractures commonly occurs  for many materials which is known as the strong discontinuity in the elasticity mechanic literature. Therefore, a known fracture  is often applied
$$
[\mathbf{u}]= (\varphi_1,\varphi_2)^T,
$$
where $\varphi_1 $ and $\varphi_2$ are fracture components for $u_1$ and $u_2$, respectively.

\subsection{Elasticity interface problem}
Based on the above discussion, two dimensional static interface problems can be formulated as:
\begin{eqnarray}
     \nabla\cdot\mathbb{T}+\mathbf{F} &=& \mathbf{0}, \quad \mbox{in} \quad  \Omega^+\cup\Omega^-,\\
    \left[\mathbf{u}\right] &=& \mathbf{b}, \quad \mbox{on} \quad \Gamma,\\
    \left[\mathbb{T}\cdot\mathbf{n}\right] &=& \mathbf{T}, \quad \mbox{on} \quad \Gamma,\\
    \mathbf{u} &=& \mathbf{u}^0, \quad  \mbox{on} \quad \partial\Omega.
\end{eqnarray}
where $\mathbf{u}({\bf x})=(u_1({\bf x}), u_2({\bf x}))^T: \Omega\rightarrow \mathbb{R}^2$ with ${\bf x}=(x, y)$ is the displacement field. Generally, if vector $\mathbf{b}$ does not equal 0, it is called the strong discontinuity. Otherwise, we arrived at the weak discontinuity. Vector $\mathbf{n}=(n_1, n_2)^T$ is the unit outer normal vector of the interface
$\Gamma$, and $\mathbf{T}=(\phi, \psi)^T$ is a given vector valued function on the interface $\Gamma$ which
measures the jump of the traction $\mathbb{T}\cdot \mathbf{n}$ across the interface $\Gamma$.
Here $\mathbf{F}$ is a 2D vector-valued function which denotes the body force on the elastic object, and
$\mathbf{u}^0=(u_{1}^0, u_{2}^0)$ is also a 2D vector-valued function to determine the Dirichlet boundary condition.

For isotropic elasticity problems,  the stress-strain relation is given by
\begin{equation}
\label{strain_stress}
\mathbb{T}=\lambda {\rm tr}(\sigma)I+2\mu\sigma,
\end{equation}
where
$$
\sigma=\frac{1}{2}\left(\nabla \mathbf{u}+(\nabla \mathbf{u})^T\right),
$$
is the linear strain, and $\lambda$ and $\mu$ are two   Lam\'{e}'s parameters, which can be constants or spatially dependent functions.
These two situations are discussed below.

\subsubsection{General formulation for homogeneous media}

A special case is that, all the material parameters are constants or piecewise constants, i.e., Lam\'{e}'s parameters
$\lambda$ and $\mu$, Young modulus $E$, and Poisson's ratio $\nu$ are all constants or piecewise constants. In this case, the above
parameters satisfies the following relationships
$$
\mu=\f{E}{2(1+\nu)},\qquad\lambda=\f{E\nu}{(1+\nu)(1-2\nu)}.
$$
%

Particularly, we assume that   shear modulus and Poisson's ratio are given, respectively, by:
\begin{equation}
\label{shmo1}
\mu=
\left\{\begin{array}{ll}
\mu^+, &\ \  \mbox{in}\ \ \Omega^+,\\
\mu^-, &\ \  \mbox{in}\ \ \Omega^-.
\end{array}\right. , \quad \nu=
\left\{\begin{array}{ll}
\nu^+, &\ \  \mbox{in}\ \ \Omega^+,\\
\nu^-, &\ \  \mbox{in}\ \ \Omega^-.
\end{array}\right. .
\end{equation}
Usually, Poisson's ratios satisfy constraints $\nu^+$ and $\nu^- \in (0, 0.5)$.  When they are close to $0.5$, the material is near  incompressible. Otherwise,
the material  is compressible.

The governing equations for the linear elasticity motion in the homogeneous media is given by:
\begin{eqnarray}\label{el1}
&&2(1-\nu)\f{\pa^2u_1}{\pa x^2}+(1-2\nu)\f{\pa^2u_1}{\pa y^2}+\f{\pa^2 u_2}{\pa x\pa y}   = -\frac{F_1}{\mu+\lambda}, \quad \mbox{on} \quad  \Omega^+\cup\Omega^-, \\ \label{el2}
&&(1-2\nu)\f{\pa^2 u_2}{\pa x^2}+2(1-\nu)\f{\pa^2 u_2}{\pa y^2}+\f{\pa^2 u_1}{\pa x\pa y} = -\frac{F_2}{\mu+\lambda}, \quad \mbox{on} \quad  \Omega^+\cup\Omega^-,
\end{eqnarray}
with the interface jump condition defined on $\Gamma$
\begin{eqnarray}\label{el3}
&&[u_1]|_\Gamma = 0, \\ \label{el4}
&&[u_2]|_\Gamma = 0,  \\ \label{el5}
&&\left[\f{2\mu}{1-2\nu}\left((1-\nu)\f{\pa u_1}{\pa x}+\nu\f{\pa u_2}{\pa y}\right)n_1
+\mu(\f{\pa u_1}{\pa y}+\f{\pa u_2}{\pa x})n_2\right]|_\Gamma = \phi,  \\\label{el6}
&&\left[\mu(\f{\pa u_1}{\pa y}+\f{\pa u_2}{\pa x})n_1+\f{2\mu}{1-2\nu}\left(\nu\f{\pa u_1}{\pa x}+
(1-\nu)\f{\pa u_2}{\pa y}\right)n_2\right]|_\Gamma = \psi.
\end{eqnarray}

\subsubsection{General formulation for inhomogeneous media}

Spatially dependent  Lam\'{e}'s parameters  frequently occur in many practical applications. The spatial dependence can be described as
\begin{equation}
\label{shmo}
\mu=
\left\{\begin{array}{ll}
\mu^+(\mathbf{x}), &\ \  \mbox{in}\ \ \Omega^+,\\
\mu^-(\mathbf{x}), &\ \  \mbox{in}\ \ \Omega^-.
\end{array}\right. , \quad
\lambda=
\left\{\begin{array}{ll}
\lambda^+(\mathbf{x}), &\ \  \mbox{in}\ \ \Omega^+,\\
\lambda^-(\mathbf{x}), &\ \  \mbox{in}\ \ \Omega^-,
\end{array}\right. .
\end{equation}

Linear elasticity motion in inhomogeneous media is governed by:
\begin{eqnarray} \label{el12}
(\lambda+2\mu)\f{\partial^2 u_1}{\partial x^2}+\mu\f{\partial^2 u_1}{\partial y^2}+(\lambda+\mu)\f{\partial^2 u_2}{\partial x \partial y}
+(\lambda_x+2\mu_x)\f{\partial u_1}{\partial x}+\lambda_x \f{\partial u_2}{\partial y}+\mu_y\f{\partial u_1}{\partial y}+ \mu_y\f{\partial u_2}{\partial x} =  -F_1, \\ \label{el22}
\mu\f{\partial u_2}{\partial x^2}+(\lambda+2\mu)\f{\partial u_2}{\partial^2 y^2}+(\lambda+\mu)\f{\partial^2 u_1}{\partial x \partial y}+
\mu_x\f{\partial u_1}{\partial y}+ \mu_x\f{\partial u_2}{\partial x}+ (\lambda_y+2\mu_y)\f{\partial u_2}{\partial y}+\lambda_y \f{\partial u_1}{\partial x}
= -F_2.
\end{eqnarray}
Here these two equations are defined in the domain $\Omega^+\cup\Omega^-$.






\begin{remark}
As discussed earlier, the non-fracture conditions can be relaxed in both theoretical  modeling and numerical analysis
\begin{equation}
\label{el32x}
[u_1]|_\Gamma=\varphi_1, \ \mbox{on}\ \Gamma,
\end{equation}
and
\begin{equation}
\label{el42y}
[u_2]|_\Gamma=\varphi_2, \ \mbox{on}\ \Gamma.
\end{equation}
\end{remark}

\section{Methods and algorithms}\label{algorithm}

In this section we develop the second order MIB method for the elasticity interface problem with irregular interface.
We consider a rectangular domain  $\Omega=[a, b]\times[c, d]$.  Let $h$ be the grid size in griding the rectangular domain and
suppose the grid points to be:
$$
x_i=a+(i-1)h,\ \  y_j=c+(j-1)h,\ \  i=1, 2, \cdots,n_x;\ \ j=1, 2, \cdots,n_y.
$$
Here $n_x$ and $n_y$ are the total numbers of grid points in the $x$- and $y$-directions, respectively.
In the standard central finite difference(CFD) discretizations scheme, due to the existence of interface and possibly discontinuous of the solution, the direct use of the CFD scheme may decrease the numerical accuracy, which cannot guarantee the second order convergence of the numerical solution.
The main idea of the MIB method, is to replace the referred function values that from the different side of the interface in the (CFD) discretizations schemes 
by the fictitious values, which are the combinations of function values and interface conditions for capturing the interface conditions and discontinuity of the solutions.

In order to handle the interface in the elasticity interface problem, the regular and irregular grid points in the finite difference scheme
of the elasticity equation should be distinguished.

\begin{definition}
A grid point $(i, j)$ is said to be a regular grid point provided all the grid points referred in the discretization of the elasticity
equations are at the same sub-domain as grid point $(i, j)$; otherwise the grid point is called an irregular grid point.
\end{definition}

For the irregular grid point $(i, j)$, if the function values at the different side of the interface applied to the finite difference scheme,
the accuracy of the numerical solution will be reduced. In the MIB finite difference
scheme, we replace function values which are referred in the different side of the interface with fictitious values.

In the following part of this section we discuss how to extend the solution across the interface such that the second order MIB discretization works. To clarify, we call the discretizations of ordinary derivatives $(u_1)_{xx}, (u_1)_{yy}, (u_2)_{xx} $ and $  (u_2)_{yy}$ as a five-point stencil, and the discretizations of the cross derivatives as a nine-point stencil. We discuss how to determine fictitious values in these two stencils separately.

\subsection{General algorithms for fictitious value}\label{ordinary-derivative}

In this section, we develop the MIB  method for  the  discretization of  the ordinary derivatives in elasticity equations via fictitious values. Here ordinary derivatives refer to first or second order derivatives along a single direction.  In contrast, a cross derivative involves at least two directions  and requires special treatments near the interface. For simplicity, we only  construct the scheme for the constant material parameter case, while the case of spatially dependent parameters can be treated in the same manner.

\subsubsection{Fictitious scheme for regular interface}
The interface is said to be regular at the grid point $(i, j)$ if the interface is locally parallel or vertical to the mesh lines that passing through the grid point $(i, j)$. For the regular interface, we can determine  fictitious values via an iteratively approach, which ensures that the MIB method can be made into arbitrarily high order.  A detail description of the MIB procedure has been given for  iteratively  finding fictitious values for regular interface \cite{Zhou:2006c}.

In this section, we focus on the method for estimating  fictitious values for a five-point stencil in the second order finite difference scheme of elasticity equations.
\begin{figure}[!ht]
\small
\centering
\includegraphics[width=6cm,height=5cm]{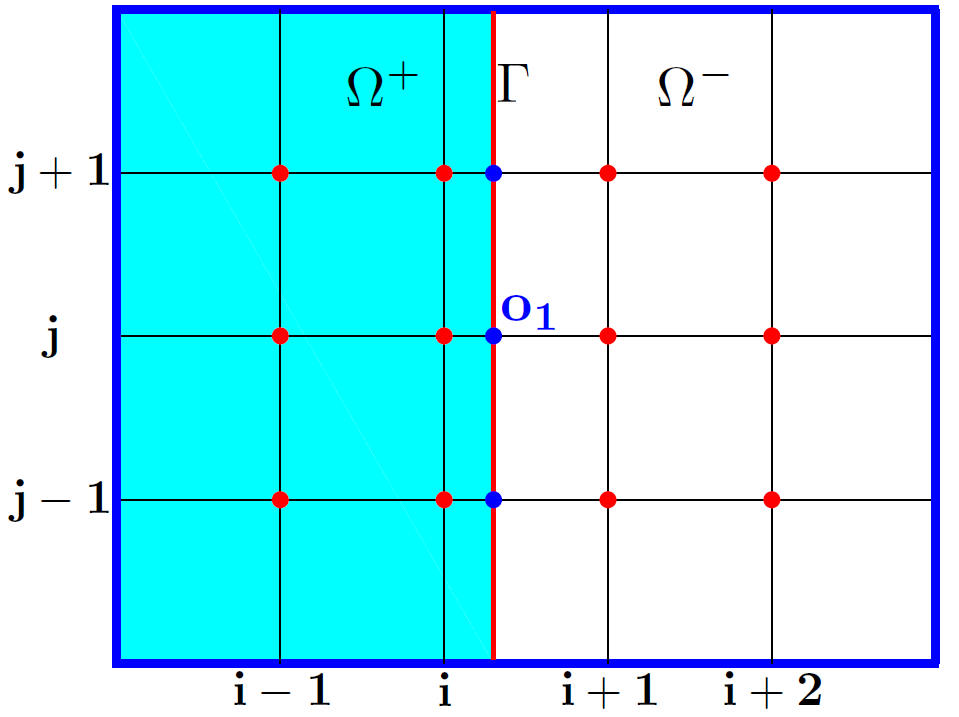}
\caption{Illustration of a regular interface. The interface (red line) is vertical to the horizontal mesh line. In the standard second order MIB finite difference scheme, fictitious values are required at  grid points $(i, j)$ and $(i+1, j)$.  To determine the fictitious value at $(i, j)$ and $(i+1, j)$,  one needs to compute $\left(\mathbf{u}^+\right)_{o_1}$ $\left(\mathbf{u}^-\right)_{o_1}$  $\left( \frac{\partial \mathbf{u}^+}{\partial x}\right)_{o_1}$ and  $\left( \frac{\partial \mathbf{u}^-}{\partial x}\right)_{o_1}$. To this end, $\left(\mathbf{u}^+\right)_{o_1}$ and 
$\left( \frac{\partial \mathbf{u}^+}{\partial x}\right)_{o_1}$ are interpolated by the function values at $(i-1, j), (i, j)$ and fictitious values at $(i+1, j)$; $\left(\mathbf{u}^-\right)_{o_1}$ and $\left( \frac{\partial \mathbf{u}^-}{\partial x}\right)_{o_1}$ are interpolated by the function values at $(i+1, j), (i+2, j)$ and fictitious values at $(i, j)$.
}
\label{iter_reg}
\end{figure}

As is shown in Figure \ref{iter_reg}, the interface is  locally vertical to the $x$-mesh at point $o_1$. Here we only
discuss the case of a vertical regular interface, while the case of a parallel regular interface can be handled in the same manner.

We employ the second order accurate interpolation scheme to interpolate one-sided values at point $o_1$. Four fictitious values are involved, i.e., $f_1^c(i, j), f_2^c(i, j), f_1^c(i+1, j) $ and $ f_2^c(i+1, j)$, here $f_1^c(i, j)$ represents the fictitious value for $u_1$ at the grid point $(i, j)$, similar for others. In the following, we present the method to represent fictitious values by interface conditions and function values.

First, the outer normal direction $\mathbf{n}=(1, 0)^T$, when the interface is locally vertical to the $x$-mesh, the interface conditions (\ref{el3})-(\ref{el6}) become
\begin{eqnarray}\label{in1}
&&u^+_1-u^-_1 = 0, \\ \label{in2}
&&u^+_2-u^-_2 = 0,  \\ \label{in3}
&&\left(\frac{2\mu^+(1-\nu^+)}{1-2\nu^+}\frac{\pa u^+_1}{\pa y} +\frac{2\mu^+\nu^+}{1-2\nu^+}\frac{\pa u^+_2}{\pa y}\right)
-\left(\frac{2\mu^-(1-\nu^-)}{1-2\nu^-}\frac{\pa u^-_1}{\pa y} +\frac{2\mu^-\nu^-}{1-2\nu^-}\frac{\pa u^-_2}{\pa y}\right)=\phi,  \\\label{in4}
&&\left(\mu^+\frac{\pa u^+_1}{\pa y}+\mu^+\frac{\pa u^+_2}{\pa x}\right)-\left(\mu^-\frac{\pa u^-_1}{\pa y}+\mu^-\frac{\pa u^-_2}{\pa x}\right)=\psi.
\end{eqnarray}
The second order interpolation approximation of $(u^+_1)_{o_1}$, $(u^-_1)_{o_1}$, $(\frac{\pa u^+_1}{\pa x})_{o_1}$ and
$(\frac{\pa u^-_1}{\pa x})_{o_1}$ can be done by using the function and fictitious values at the grid points $(i-1, j)$,
$(i, j)$ and $(i+1, j)$
\begin{eqnarray}
\left(u^+_1\right)_{o_1}&=&w^+_{0, 1}u_1(i-1, j)+w^+_{0, 2}u_1(i, j)+w^+_{0, 3}f_1^c(i+1, j),\\
\left(u^-_1\right)_{o_1}&=&w^-_{0, 1}f_1^c(i, j)+w^-_{0, 2}u_1(i+1, j)+w^-_{0, 3}u_1(i+2, j), \\
\left(\frac{\pa u^+_1}{\pa x}\right)_{o_1}&=&w^+_{1, 1}u_1(i-1, j)+w^+_{1, 2}u_1(i, j)+w^+_{1, 3}f_1^c(i+1, j), \\
\left(\frac{\pa u^-_1}{\pa x}\right)_{o_1}&=&w^-_{1, 1}f_1^c(i, j)+w^-_{1, 2}u_1(i+1, j)+w^-_{1, 3}u_1(i+2, j),
\end{eqnarray}
where $\{w^+_{m, n} \mid m=0, 1; n=1, 2, 3\}$ are the coefficients in the Lagrangian interpolation and can be obtained by an appropriate algorithm in the literature \cite{fonberg:1998}. The interpolation approximation of $(u^+_2)_{o_1}$, $(u^-_2)_{o_1}$, $(\frac{\pa u^+_2}{\pa x})_{o_1}$ and $(\frac{\pa u^-_2}{\pa x})_{o_1}$ can be handled in the same manner.

Plugging the above finite difference approximation into  interface conditions (\ref{in1})-(\ref{in4}) and solving the resulting
approximation of  interface conditions, one can represent fictitious values at grid points $(i, j)$ and $(i+1, j)$
by function values at grid points $(i-1, j)$, $(i, j)$, $(i+1, j)$ and  $(i+2, j)$, and interface conditions at $o_1$.

\subsubsection{Fictitious scheme for interface with small curvature}

We describe the MIB method  for the $C^1$ extension of the function values across the interface which is irregular but with
relatively small curvatures in this subsection.
Interface conditions  are given by Eqs. (\ref{el3})-(\ref{el6}).




Figure \ref{fic_x} illustrates the situation  where fictitious values along the $x$-direction are to be found. The case for determining fictitious values along  the $y$-direction is similar.

\begin{figure}[!ht]
\small
\centering
\includegraphics[width=6cm,height=5cm]{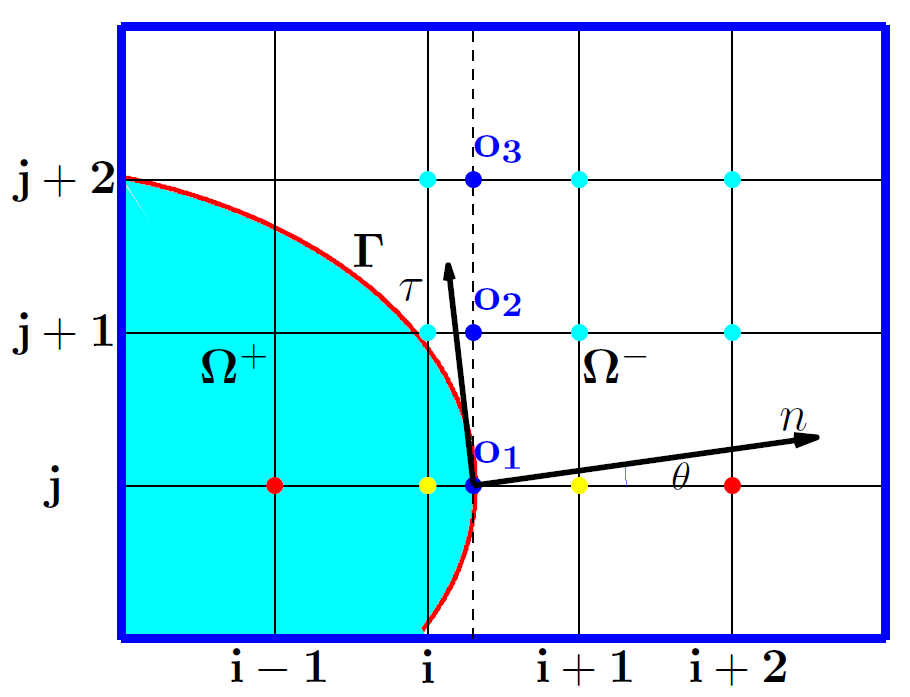}
\caption{Illustration of determining fictitious values along the $x$-direction. Fictitious values at two yellow grid point are to be found.  To this end, derivatives at $o_1$ need to be calculated.  Two red points ($(i, j)$ and $(i+1, j)$) are the points referred along the main direction. Six cyan points are auxiliary points used to approximate values at $o_2$ and $o_3$. Values at $o_1$ can be computed from the fictitious value at $(i,j)$ and function values at $(i+1,j)$ and $(i+2,j)$.  Then interpolated values at $o_1, o_2$ and $o_3$ are employed to estimate derivatives at $o_1$.}
\label{fic_x}
\end{figure}

When the interface is locally not aligned to an $x$-mesh line or $y$-mesh line, a  two-step tactic in the MIB approach can be used.  First,  one takes the derivative along the tangential direction of the interface to introduce one more set of interface conditions. Then one can eliminate some derivatives at the interface that are difficult to compute. This flexibility directly leads to the efficiency in handling versatile and difficult interface geometries. Finally, one obtains the representation of fictitious values according to interface conditions after the aforementioned two steps.

\begin{lemma}
For a given function $u$ defined on the domain $\Omega=\Omega^+\cup \Gamma \cup \Omega^-$, i.e., the domain is separated by
the interface $\Gamma \in C^1$ into two parts. If the function $u$ is piecewise $C^1$ continuous along the interface $\Gamma$, then its tangential derivative is continuous along the interface.
\end{lemma}

Let the tangential direction be $\mathbf{\tau}=(\tau_1, \tau_2)$. By taking a derivative along the tangential direction,   the obtained new set of  interface jump conditions along the tangential direction are:


\begin{equation}
\label{el522}
\left[\frac{\pa u_1}{\pa x}\tau_1+\frac{\pa u_1}{\pa y}\tau_2\right]|_\Gamma=0,
\end{equation}
and
\begin{equation}
\label{el622}
\left[\frac{\pa u_2}{\pa x}\tau_1+\frac{\pa u_2}{\pa y}\tau_2\right]|_\Gamma=0,
\end{equation}


As shown in Figure \ref{fic_x}, let the angle between the normal direction and $x$-mesh be $\theta$, in this case, the normal
direction is $\mathbf{n}=(\cos\theta, \sin\theta)$, and the tangential direction is $\mathbf{\tau}=(-\sin\theta, \cos \theta)$, where
$0 \leq \theta \leq 2\pi$. For the irregular interface, one has  $\theta \neq 0, 0.5\pi, 1.0\pi, 1.5\pi $ and $ 2\pi$.

If we define a vector $C$ as,
\begin{eqnarray}
     C=\left(\frac{\pa u^+_1}{\pa x},\frac{\pa u^-_1}{\pa x},\frac{\pa u^+_1}{\pa y},\frac{\pa u^-_1}{\pa y},\frac{\pa u^+_2}{\pa x},\frac{\pa u^-_2}{\pa x},\frac{\pa u^+_2}{\pa y},\frac{\pa u^-_2}{\pa y}\right)^T.
\end{eqnarray}
The interface condition along the tangential direction (\ref{el522})-(\ref{el622}) can be represented as,
\begin{eqnarray}
 &&   (-\sin\theta, \sin\theta, \cos\theta, -\cos\theta, 0, 0, 0, 0)\cdot C=0, \\
 &&  (0 , 0, 0, 0, -\sin\theta, \sin\theta, \cos\theta, -\cos\theta)\cdot C=0.
\end{eqnarray}
Furthermore, additional interface conditions along the normal direction can be written as
\begin{eqnarray}
 &&   (M^+\cos\theta, -M^-\cos\theta, \mu^+\sin\theta, -\mu^-\sin\theta, \mu^+\sin\theta, -\mu^-\sin\theta, \lambda^+ \cos\theta, -\lambda^- \cos\theta)\cdot C=\phi, \\
 &&  (\lambda^+\sin\theta, -\lambda^-\sin\theta, \mu^+ \cos\theta, -\mu^-\cos\theta, \mu^+ \cos\theta, -\mu^-\cos\theta, M^+\sin\theta,-M^-\sin\theta )\cdot C=\psi.
\end{eqnarray}
where $M=\frac{2\mu(1-\nu)}{1-2\nu}$ is the $p-$wave modulus.

In the standard second-order MIB finite difference scheme, two fictitious values occur for each of  elasticity equations near the interface. A total of four fictitious values are to be estimated. Since there are six interface conditions (\ref{el3})-(\ref{el6}) and (\ref{el522})-(\ref{el622}), one can use  them to eliminate two of eight derivatives.
The selection of the derivatives should be based on the local geometry of the interface.  A main principle is to eliminate the derivatives that are most difficult to compute due to the local geometric constraint.

\begin{lemma}
The following matrix is of full rank,
$$
\left(
  \begin{array}{cccccccc}
    -\sin\theta & \sin\theta & \cos\theta & -\cos\theta & 0 & 0 & 0 &0 \\
    0 & 0& 0& 0& -\sin\theta & \sin\theta & \cos\theta &-\cos\theta\\
    M^+\cos\theta & -M^-\cos\theta & \mu^+\sin\theta & -\mu^-\sin\theta & \mu^+\sin\theta &-\mu^-\sin\theta &\lambda^+\cos\theta &-\lambda^-\cos\theta\\
    \lambda^+\sin\theta & -\lambda^-\sin\theta & \mu^+\cos\theta & -\mu^-\cos\theta & \mu^+\cos\theta &-\mu^-\cos\theta &M^+\sin\theta &-M^-\sin\theta
  \end{array}
\right)
$$
provided that  $\theta \neq 0, 0.5\pi, \pi, 1.5\pi $ and $ 2\pi$, i.e.,
the local interface does not parallel or vertical to the local mesh directions.
\end{lemma}

According to the above lemma, when the local interface is irregular, the third and fourth rows of the above matrix should be used
to minus the linear combinations of the first and second rows so as to replace two interface conditions and leave two interface conditions
to compute the fictitious values.

There are four ways to eliminate  derivatives at the interface as  discussed below.
\begin{itemize}
\item The elimination of the derivatives $\frac{\pa u^-_1}{\pa y}$ and $\frac{\pa u^-_2}{\pa y}$ generates the following combined interface conditions:
\begin{eqnarray}\nonumber
&&-\phi \cos\theta = \left(-M^+\cos^2\theta-\mu^-\sin^2\theta, M^-\cos^2\theta+\mu^-\sin^2\theta, (\mu^--\mu^+)\sin\theta\cos\theta,  0 \right.\\
&&\qquad \qquad \left. -(\lambda^-+\mu^+)\sin\theta\cos\theta, (\lambda^-+\mu^-)\cos\theta\sin\theta, (\lambda^--\lambda^+)\sin\theta\cos\theta, 0 \right)\cdot C,\\ \nonumber
&&-\psi \cos\theta = \left(-(\lambda^++\mu^-)\sin\theta\cos\theta, (\lambda^-+\mu^-)\cos\theta\sin\theta, (\mu^--\mu^+)\cos^2\theta, 0 \right. \\
&&\qquad \qquad  \left.-\mu^+\cos^2\theta-M^-\sin^2\theta, \mu^-\cos^2\theta+M^-\sin^2\theta, (M^--M^+)\sin\theta\cos\theta, 0 \right)\cdot C .
\end{eqnarray}

\item The elimination of the derivatives $\frac{\pa u^+_1}{\pa y}$ and $\frac{\pa u^+_2}{\pa y}$ generates the following combined interface conditions:
\begin{eqnarray}\nonumber
&&\phi\cos\theta = \left(M^+\cos^2\theta+\mu^+\sin^2\theta, -(M^-\cos^2\theta+\mu^+\sin^2\theta) , 0,
(\mu^+-\mu^-)\sin\theta\cos\theta,\right.\\
&&\qquad \qquad \left.(\mu^++\lambda^+)\sin\theta\cos\theta,-(\mu^-+\lambda^+)\sin\theta\cos\theta , 0, (\lambda^+-\lambda^-)\cos^2\theta \right)\cdot C,\\ \nonumber
&&\psi \cos\theta = \left((\lambda^++\mu^+)\sin\theta\cos\theta, -\lambda^-\sin^2\theta-\mu^+\sin\theta\cos\theta, 0,
-(\mu^+-\mu^-)\cos^2\theta, \right. \\
&&\qquad \qquad  \left.\mu^+\cos^2\theta+M^+\sin^2\theta, -\mu^-\cos^2\theta-M^+\sin^2\theta, 0, (M^+-M^-)\sin\theta\cos\theta \right)\cdot C .
\end{eqnarray}

\item The elimination of the derivatives $\frac{\pa u^-_1}{\pa x}$ and $\frac{\pa u^-_2}{\pa x}$ generates the following combined interface conditions:
\begin{eqnarray}\nonumber
&&\phi\sin\theta = \left((M^+-M^-)\sin\theta\cos\theta, 0, \mu^+\sin^2\theta+M^-\cos^2\theta, -\mu^-\sin^2\theta-M^-\cos^2\theta, \right.\\
&&\qquad \qquad \left.(\mu^+-\mu^-)\sin^2\theta,0, (\lambda^++\mu^-)\sin\theta\cos\theta, -(\lambda^-+\mu^-)\sin\theta\cos\theta  \right)\cdot C ,\\ \nonumber
&&\psi\sin\theta = \left((\lambda^+-\lambda^-)\sin^2\theta, 0, (\mu^++\lambda^-)\sin\theta\cos\theta, -(\mu^-+\lambda^-)\sin\theta\cos\theta,\right. \\
&&\qquad \qquad  \left.(\mu^+-\mu^-)\sin\theta\cos\theta, 0, M^+\sin^2\theta+\mu^-\cos^2\theta, -M^-\sin^2\theta-\mu^-\cos^2\theta \right)\cdot C .
\end{eqnarray}

\item The elimination of the derivatives $\frac{\pa u^+_1}{\pa x}$ and $\frac{\pa u^+_2}{\pa x}$ generates the following combined interface conditions:
\begin{eqnarray}\nonumber
&&-\phi\sin\theta = \left(0, (M^--M^+)\sin\theta\cos\theta, -(\mu^+\sin^2\theta+M^+\cos^2\theta), \mu^-\sin^2\theta+M^+\sin^2\theta,\right.\\
&&\qquad \qquad \left.0, (\mu^--\mu^+)\sin^2\theta, -(\lambda^++\mu^+)\sin\theta\cos\theta, (\lambda^-+\mu^+)\sin\theta\cos\theta
\right)\cdot C ,\\ \nonumber
&&-\psi\sin\theta = \left(0, (\lambda^++\lambda^-)\sin^2\theta, -\mu^+\sin\theta\cos\theta-\lambda^+\sin^2\theta, (\mu^--\lambda^+)\sin\theta\cos\theta, \right. \\
&&\qquad \qquad  \left.0, (\mu^-+\mu^+)\sin\theta\cos\theta, -M^+\sin^2\theta-\mu^+\cos^2\theta, M^-\sin^2\theta+\mu^+\cos^2\theta\right)\cdot C .
\end{eqnarray}
\end{itemize}


For a  given irregular grid point, according to the local interface geometry, one of the replaced interface conditions
from above should be chosen to represent  fictitious values. Here, we discuss the interface shown in Figure \ref{fic_x}.
In this case,  one-sided derivatives $\frac{\pa u^+_1}{\pa y}$ and $\frac{\pa u^+_2}{\pa y}$ are  to be eliminated. Therefore,
the first set of interface conditions is employed. Note that the interpolation approximation of one-sided function values
and derivatives referred in the first set of interface conditions for $u_1$ is given as:
$$
u^-_1=(f_1^c(i, j), u_1(i+1, j), u_1(i+2, j))\cdot(w_{0, 0}, w_{0, 1}, w_{0, 2})^T,
$$
$$
u^+_1=(u_1(i-1, j), u_1(i, j), f_1^c(i+1, j))\cdot (\tilde{w}_{0, 0}, \tilde{w}_{0, 1}, \tilde{w}_{0, 2})^T,
$$
$$
\frac{\pa u^-_1}{\pa x}=(f_1^c(i, j), u_1(i+1, j), u_1(i+2, j))\cdot(w_{1, 0}, w_{1, 1}, w_{1, 2})^T,
$$
$$
\frac{\pa u^+_1}{\pa x}=(u_1(i-1, j), u_1(i, j), f_1^c(i+1, j))\cdot (\tilde{w}_{1, 0}, \tilde{w}_{1, 1}, \tilde{w}_{1, 2})^T,
$$
$$
\frac{\pa u^-_1}{\pa y}=(u_1(o, j), u_1(o, j+1), u_1(o, j+2))\cdot (w^*_{1, 0}, w^*_{1, 1}, w^*_{1, 2})^T,
$$
where $w_{i, j}, \tilde{w}_{i, j}$ and $w^*(1, j)$, $i=0, 1,\  j=0, 1, 2$ are the lagrangian interpolation coefficients.

Three off-grid points are referred in the approximation of $\frac{\pa u^-_1}{\pa y}$. Here
$$
u_1(o, j):=u^+_1(o, j)=u^-_1(o, j)+[u_1].
$$
Additionally,  we regard the value at the  on-interface grid point as the inside one
$$
u_1(o, j+1)=(u_1(i, j+1), u_1(i+1, j+1), u_1(i+2, j+1))\cdot(p_{0, 0}, p_{0, 1}, p_{0, 2})^T,
$$
$$
u_2(o, j+2)=(u_2(i, j+2), u_2(i+1, j+2), u_2(i+2, j+2))\cdot(p^*_{0, 0}, p^*_{0, 1}, p^*_{0, 2})^T,
$$
where $p_{0, j}, p^*_{0, j}, j=0, 1, 2$ are the Lagrangian interpolation coefficients.

The value of  $u_2$ can be approximated similarly.

Replacing the values in the first set of interface conditions by the above approximated values, and solving the generated equations  give the representation of the fictitious values that needed in the second order central finite difference schemes.

\subsubsection{Fictitious scheme for interface with large curvature}
\begin{figure}[!ht]
\small
\centering
\includegraphics[width=5cm,height=5cm]{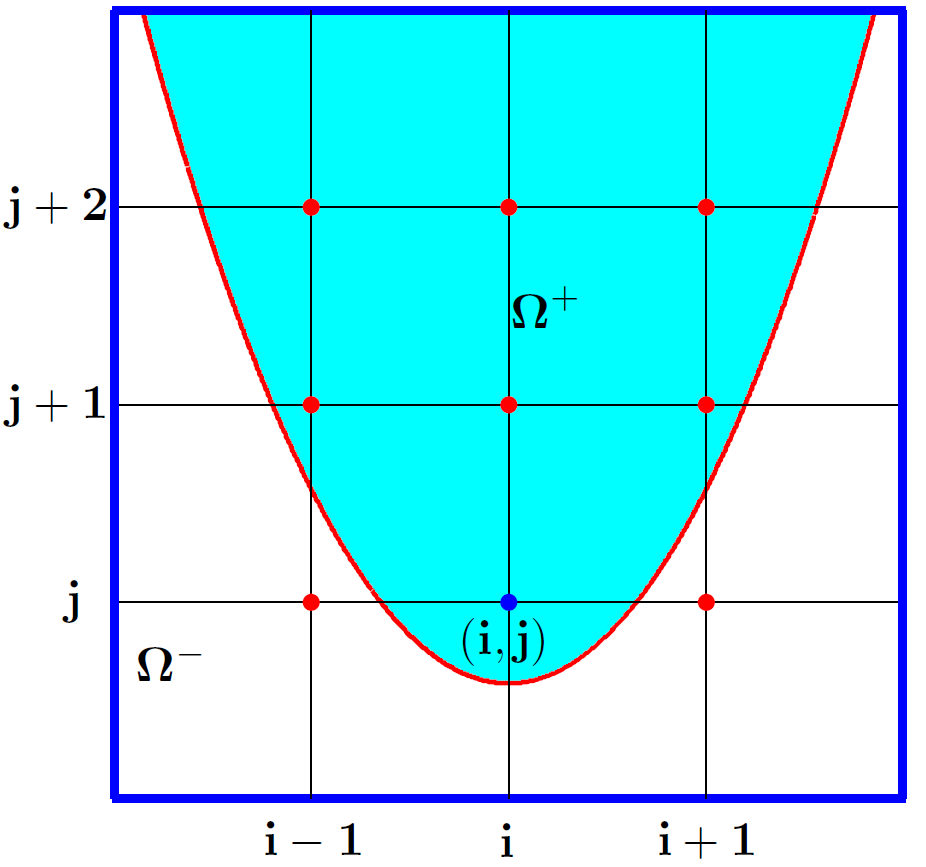}
\caption{Illustration of an interface with large curvature. In this case, the fictitious value at grid point $(i, j)$ cannot be calculated from the horizontal direction due to the lack of any inside grid point. However, it can be obtained from the vertical direction. In the discretization scheme, the fictitious value at $(i, j)$ found from the vertical direction is utilized for both vertical and horizontal discretizations of the derivatives in the governing equation.}
\label{large_5_fic}
\end{figure}

For the interface that with large curvature, the above procedure for finding fictitious values may fail at some irregular grid points. As shown in Figure \ref{large_5_fic},   the method proposed above does  not work in finding fictitious values for   irregular grid point $(i, j)$ along the $x$-direction. However, the method works for finding fictitious values at $(i, j)$ in the $y$-direction. In this case, a disassociation scheme was proposed in the MIB scheme to use fictitious values obtained along the $y$-direction to replace  fictitious values along the $x$-direction. For more detail on the disassociation scheme, reader is referred to our earlier work \cite{Zhou:2006d}. The disassociation technique retains the second-order numerical accuracy.

\subsection{Special algorithms for fictitious value for cross derivatives}

Unlike the Poisson equation which does not admit any cross derivative, elasticity equations involve cross derivatives which give rise to additional
numerical difficulties when  the interface geometry is complex.  In this section, we propose new schemes to determine fictitious values for the discretization of  cross derivatives.

First, one notes that in the discretization of the  cross derivative at a given  grid point $(i, j)$,
in addition to original five points, i.e.,  $(i, j)$,  $(i, j-1)$,  $(i, j+1)$,  $(i-1, j)$, and  $(i+1, j)$,
four more adjacent  points, namely, $(i-1, j-1)$, $(i-1, j+1)$, $(i+1, j-1)$ and  $(i+1, j+1)$ are involved in the standard central
finite difference scheme.  Therefore, more irregular points are created near the interface due to the discretization of cross derivatives.
As a result, the MIB schemes are to be extended for cross derivatives. Difficulties raise in the determination of fictitious values as the interface is complex or has large curvatures. We propose two methods, i.e., disassociation scheme and extrapolation scheme, for the determination of  fictitious values used in discretizing cross derivatives.

\subsubsection{Disassociation scheme}
To facilitate our further discussion, we classify irregular grid points into disassociation type and extrapolation type.
\begin{definition}
An irregular grid point associated with a cross derivative is called disassociation type provided that the irregular grid point is also an
irregular grid point associated with an central derivative.
\end{definition}

The fictitious values on the disassociation type of  irregular grid points can be determined by the disassociation technique proposed in our earlier work  \cite{Zhou:2006d}.

As illustrated in Figure \ref{large_5_fic},  grid point $(i, j)$ is not only irregular in central derivatives but also irregular in cross derivatives. In this circumstance, fictitious value on  $(i, j)$ can be determined by methods presented in Section \ref{ordinary-derivative}. The obtained  fictitious value can be directly utilized for the discretization of cross derivatives as well.

\subsubsection{Extrapolation scheme}
If a grid point  is irregular in the 9-point stencil while is regular in the 5-point stencil, the disassociation scheme may not work.  In this case, there are two other options can be adopted to determine  its fictitious values (note that one irregular grid point has two fictitious values because of two elasticity equations). One approach  is extrapolation, which is easy to implement and its numerical accuracy can also be maintained as shown by extensive numerical tests.  The other approach is based on an iteratively matched interface method proposed in our earlier work \cite{Zhou:2006c}. In the present work, the extrapolation scheme is developed  to determine fictitious values for the approximation of cross derivatives.

\begin{figure}
\begin{center}
\begin{tabular}{ccc}
\includegraphics[width=0.333\textwidth]{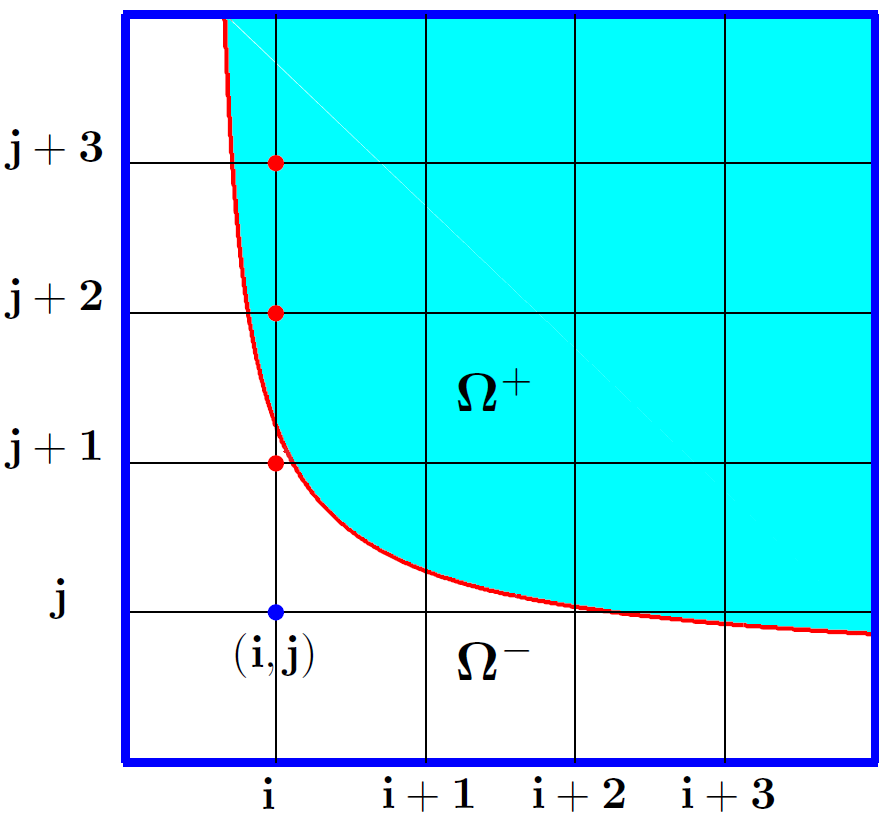}
\includegraphics[width=0.333\textwidth]{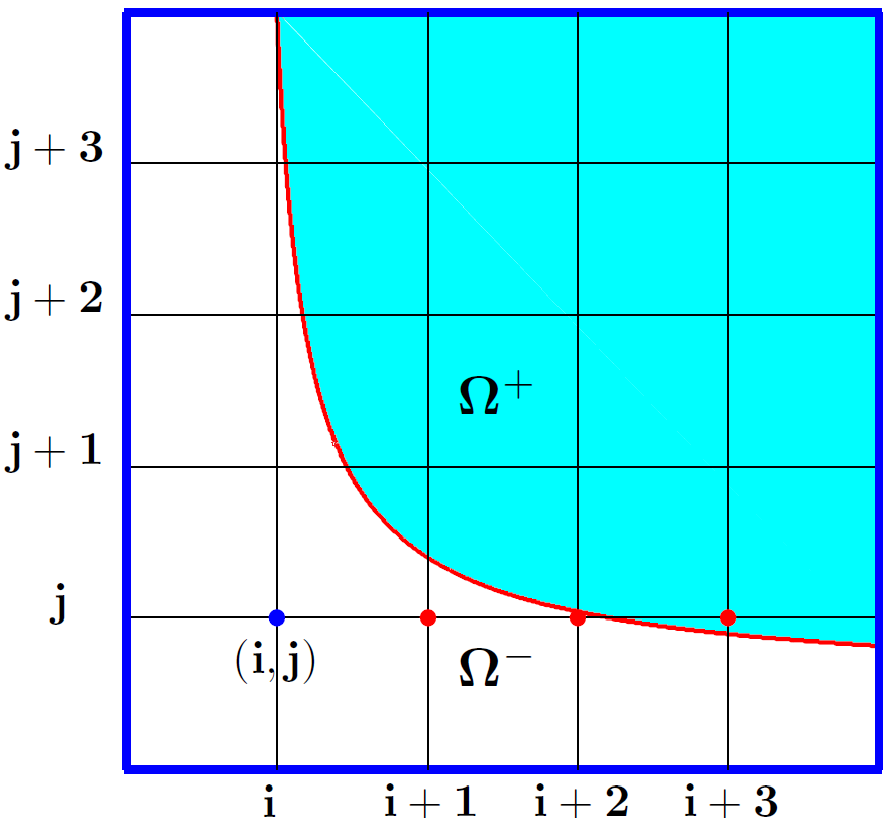}
\includegraphics[width=0.333\textwidth]{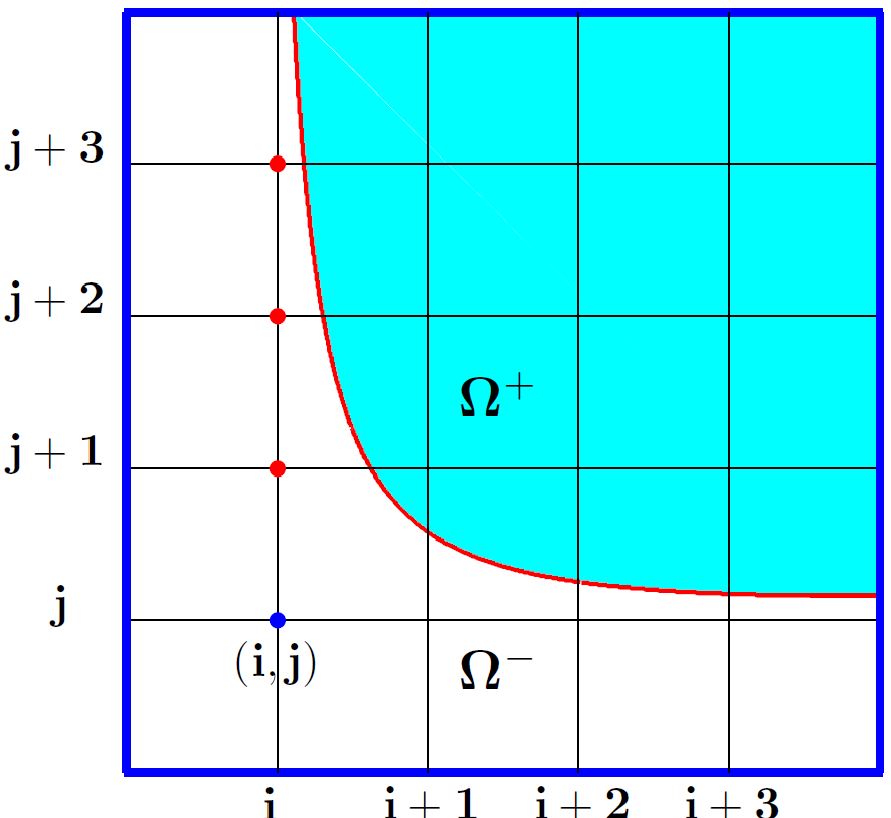}
\end{tabular}
\end{center}
\caption{Extrapolation type of irregular grid points in cross derivatives. In the left case, fictitious value at the bottom red point and function values at other two red points are employed to approximate fictitious value at $(i, j)$.  For the middle case, function values at the right-most red point and fictitious values at other two red points are utilized to extrapolate fictitious value at $(i, j)$. For the right case, fictitious values at three red points are used to approximate fictitious value at $(i, j)$.
}
\label{ext}
\end{figure}

In order to determine fictitious values at  irregular grid point $(i, j)$ for cross derivatives, we classify them into three types according to whether the function values or fictitious values are used in the extrapolation. As shown in Figures \ref{ext}, there are three types of extrapolation schemes.

\begin{itemize}
\item Scheme I. Two function values and one fictitious value are used for the extrapolation. For example, function values at $(i, j+2)$ and  $(i, j+3)$, fictitious values at $(i, j+1)$ are used to extrapolate fictitious values at $(i, j)$,  see the left chart of Fig. \ref{ext}.

\item Scheme II. One function value and two fictitious values are used for an extrapolation. For example, function values at $(i+3, j)$, fictitious values at $(i+1, j)$ and  $(i+2, j)$ used to extrapolate fictitious values at $(i, j)$,  see the middle chart of Fig. \ref{ext}.

\item Scheme III. Three fictitious values are used for an extrapolation. For example, fictitious values at grid points $(i, j+1)$, $(i, j+2)$ and  $(i, j+3)$ used to extrapolate fictitious values at $(i, j)$,  see the right chart of Fig. \ref{ext}.
\end{itemize}
In all of these schemes, extrapolations are carried out with the Lagrange polynomial.

\subsection{Governing equation discretization}
With the fictitious value represented by the function values at grid nodes and the interface conditions, we can construct the second order MIB schemes for the governing equations. The basic idea is to use fictitious values to replace the function values from the other subdomain. Since the three points interpolation or extrapolation schemes employed in matching the interface conditions, the obtained fictitious values guarantees the MIB scheme is of second order convergence. For instance, if we need to discretize the $\f{\partial^2 u_1}{\partial x \partial y}$ terms at a regular grid point $(i, j)$ in subdomain $\Omega^+$, the standard second order CFD schemes can be directly employed,
\begin{eqnarray}
\f{\partial^2 u_1}{\partial x \partial y}|_{(i, j)}=\f{1}{4h^2}(u_1(i+1, j+1)+u_1(i-1, j-1)-u_1(i-1, j+1)-u_1(i+1, j-1)).
\end{eqnarray}
However, if grid $(i, j)$ is irregular, and the nodes $(i-1, j-1)$ and $(i+1, j-1)$ are located on the other domain as depicted in Figure \ref{large_5_fic}, then the above scheme can be modified by the inclusion of the fictitious values,
\begin{eqnarray}
\f{\partial^2 u_1}{\partial x \partial y}|_{(i,j)}=\f{1}{4h^2}(u_1(i+1, j+1)+f^c_1(i-1, j-1)-f^c_1(i-1, j+1)-u_1(i+1, j-1)).
\end{eqnarray}

Similar MIB schemes can be applied for the discretization of other referred derivatives terms in the governing equations of the elasticity
interface problems (\ref{el1}-\ref{el2}) and (\ref{el12}-\ref{el22}).

\section{Numerical experiments} \label{validation}

In this section, the numerical accuracy, convergence and robustness  of the proposed second order MIB schemes for elasticity interface problems  are validated by four kinds of complex interfaces, namely, circle, ellipse, flower-like, and jigsaw-like interfaces, on rectangular domains. The first two interfaces involve  small curvatures, while the other two have   large curvatures.
The proposed numerical schemes were tested   for the piecewise constant material parameters and spatially dependent material parameters. Furthermore, to test the robustness of the present MIB method, we consider both large and small contrasts in the Poisson's ratio and shear modulus.

A standard  bi-conjugate gradient solver is used to solve the linear algebraic equations generated by the present MIB discretization. Numerical solutions are compared with the designed exact solutions.  Both $L_\infty$ and $L_2$ norm error measurements are employed in our tests and are defined as
$$
L_\infty(u_k)=\max\abs{u_k(i,j)-\tilde{u}_k(i,j)}, \quad k=1,2;  i=1,2,\cdots,n_x;  j=1,2,\cdots,n_y
$$
and
$$
L_2(u_k)=\sqrt{\f{1}{n_x*n_y}\sum_{j=1}^{n_y}\sum_{i=1}^{n_x}(u_k(i,j)-\tilde{u}_k(i,j))^2}, \quad k=1,2,
$$
where $u_k(i,j)$ are exact solutions and $\tilde{u}_k(i,j)$ are numerical solutions.

\subsection{Homogeneous media}
\subsubsection{Weak discontinuity}
In this section, various numerical tests are performed  for  cases where both Poisson's ratio and shear modulus are piecewise constant.

\paragraph{Example ~1.}
We first consider an ellipse interface defined by $x^2+4y^2=0.35^2$. In this example, the computational domain is set to $\Omega=[-0.5, 0.5]\times [-0.5, 0.5]$.   The Dirichlet boundary condition and interface conditions are determined from the following exact solution
$$
u_1(x, y)=
\left\{\begin{array}{ll}
xy+\sin(1+x^2+y^2)-3x^2+y^2, &\ \  \mbox{in}\ \ \Omega^+,\\
xy+\sin(1+x^2+y^2)-2x^2+5y^2-0.35^2, &\ \  \mbox{in}\ \ \Omega^-.
\end{array}\right.
$$
and
$$
u_2(x, y)=
\left\{\begin{array}{ll}
\cos(1+x^2-y^2)+5x^2y+x^2-y^2+2, &\ \  \mbox{in}\ \ \Omega^+,\\
\cos(1+x^2-y^2)+5x^2y+3x^2+7y^2+2(1-0.35^2), &\ \  \mbox{in}\ \ \Omega^-.
\end{array}\right.
$$

We consider three cases with different Poisson's ratio and  shear moduli.

\paragraph{Example ~1a.}
First, let us consider a example used in the literature \cite{YangXZ:2003}.
The Poisson's ratio and the shear modulus are, respectively,
$$
\nu=
\left\{\begin{array}{ll}
\nu^+=0.20, &\ \  \mbox{in}\ \ \Omega^+,\\
\nu^-=0.24, &\ \  \mbox{in}\ \ \Omega^-.
\end{array}\right.
$$
and
$$
\mu=
\left\{\begin{array}{ll}
\mu^+=1500000, &\ \  \mbox{in}\ \ \Omega^+,\\
\mu^-=2000000, &\ \  \mbox{in}\ \ \Omega^-.
\end{array}\right.
$$

Table \ref{mib_ex1} gives the grid refinement analysis and  Figure \ref{Fig.lable1} depicts our results. Obviously, the designed second order convergence is achieved for this elasticity interface problem. In the following, we further test the robustness of our method for handling large contrast between Poisson's ratio and shear modulus.

\begin{table}[!ht]
\centering
\caption{Numerical error and order for Example 1a.}
\begin{tabular}{lllllllll}
\cline{1-9}
$n_x\times n_y$ &$L_\infty(u_1)$&Order &$L_2(u_1)$&Order & $L_\infty(u_2)$&Order &$L_2(u_2)$ & Order  \\
\hline
$20\times 20$ &$4.40\times 10^{-4}$     &           &$2.44\times 10^{-4}$      &   &$2.15\times 10^{-4}$     &                 &$1.13\times 10^{-4}$      &\\
$40\times 40$ &$1.10\times 10^{-4}$     &2.00       &$6.06\times 10^{-5}$      &2.01  &$8.90\times 10^{-5}$     &1.37          &$4.14\times 10^{-5}$    &1.45\\
$80\times 80$ &$2.42\times 10^{-5}$     &2.18       &$1.31\times 10^{-5}$      &2.21  &$2.03\times 10^{-5}$     &2.13          &$9.31\times 10^{-6}$    &1.89\\
$160\times 160$ &$5.88\times 10^{-6}$     &2.04       &$3.19\times 10^{-6}$      &2.04  &$4.83\times 10^{-6}$     &2.07          &$2.28\times 10^{-6}$     &2.03\\
$160\times 160$ &$1.38\times 10^{-6}$     &2.09       &$7.80\times 10^{-7}$      &2.03  &$1.07\times 10^{-6}$     &2.17         &$5.20\times 10^{-7}$      &2.13\\
\hline
\end{tabular}
\label{mib_ex1}
\end{table}

\begin{figure}[!ht]
\centering
\subfigure[Numerical solution $u_1$]{
\label{Fig.sub.1}
\includegraphics[width=0.45\textwidth]{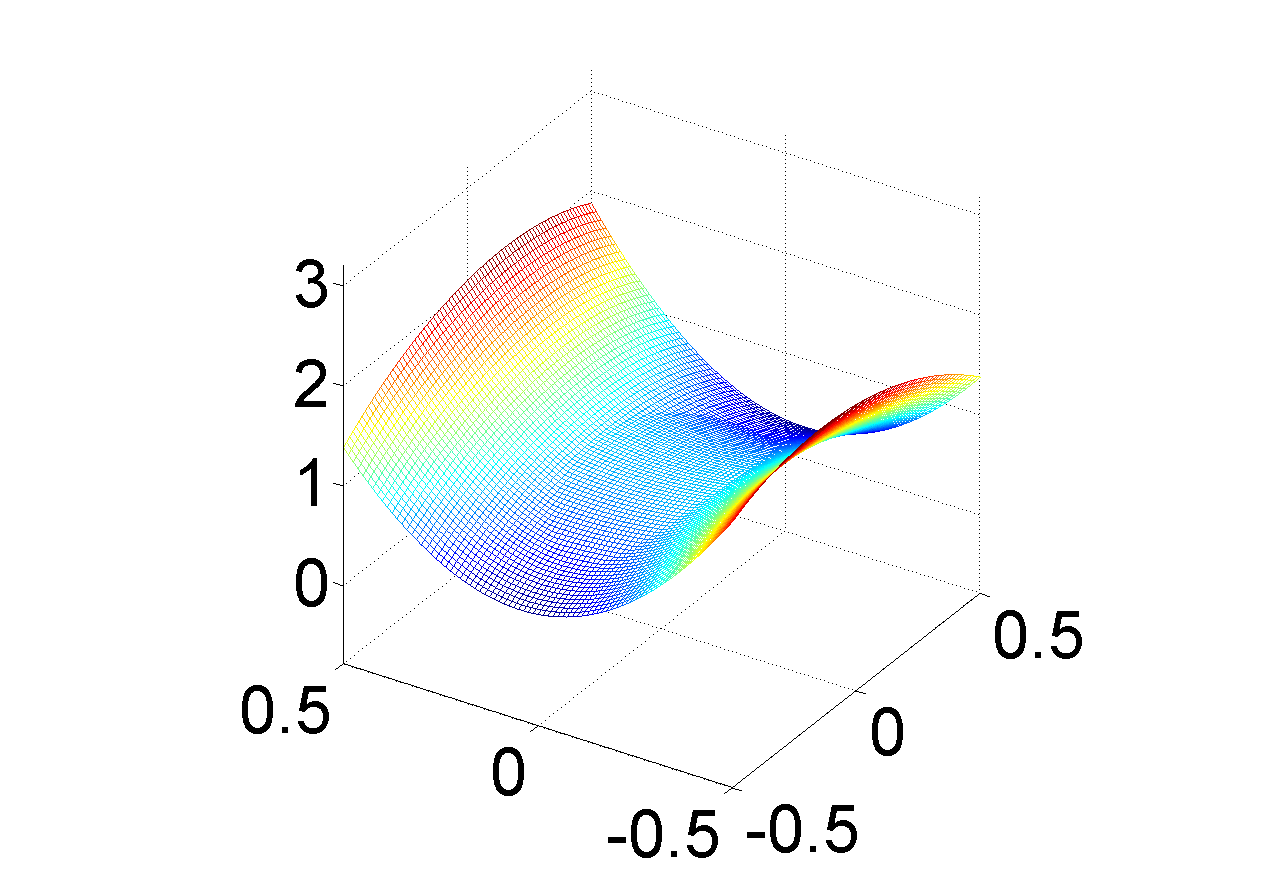}}
\subfigure[Error $u_1$]{
\label{Fig.sub.2}
\includegraphics[width=0.45\textwidth]{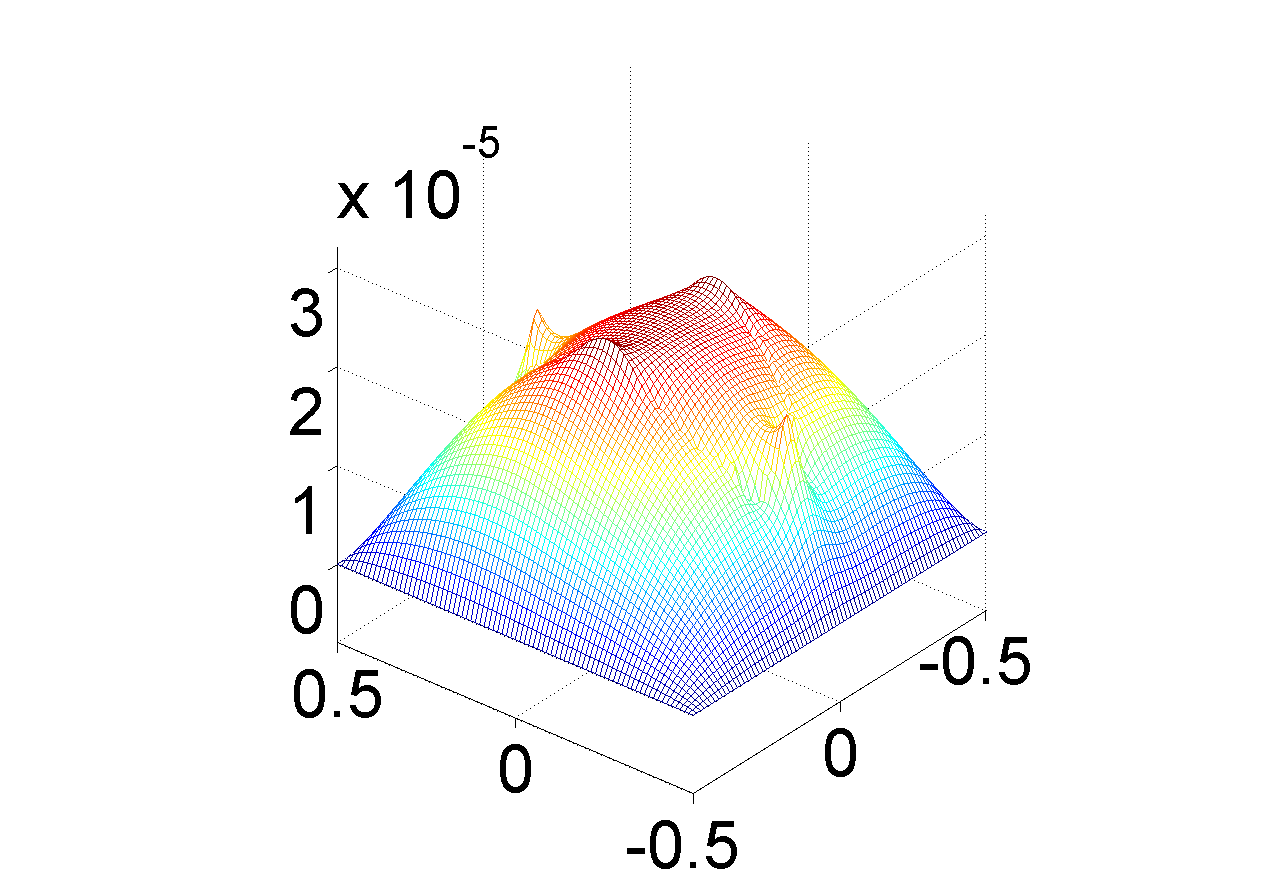}}
\subfigure[Numerical solution $u_2$]{
\label{Fig.sub.3}
\includegraphics[width=0.45\textwidth]{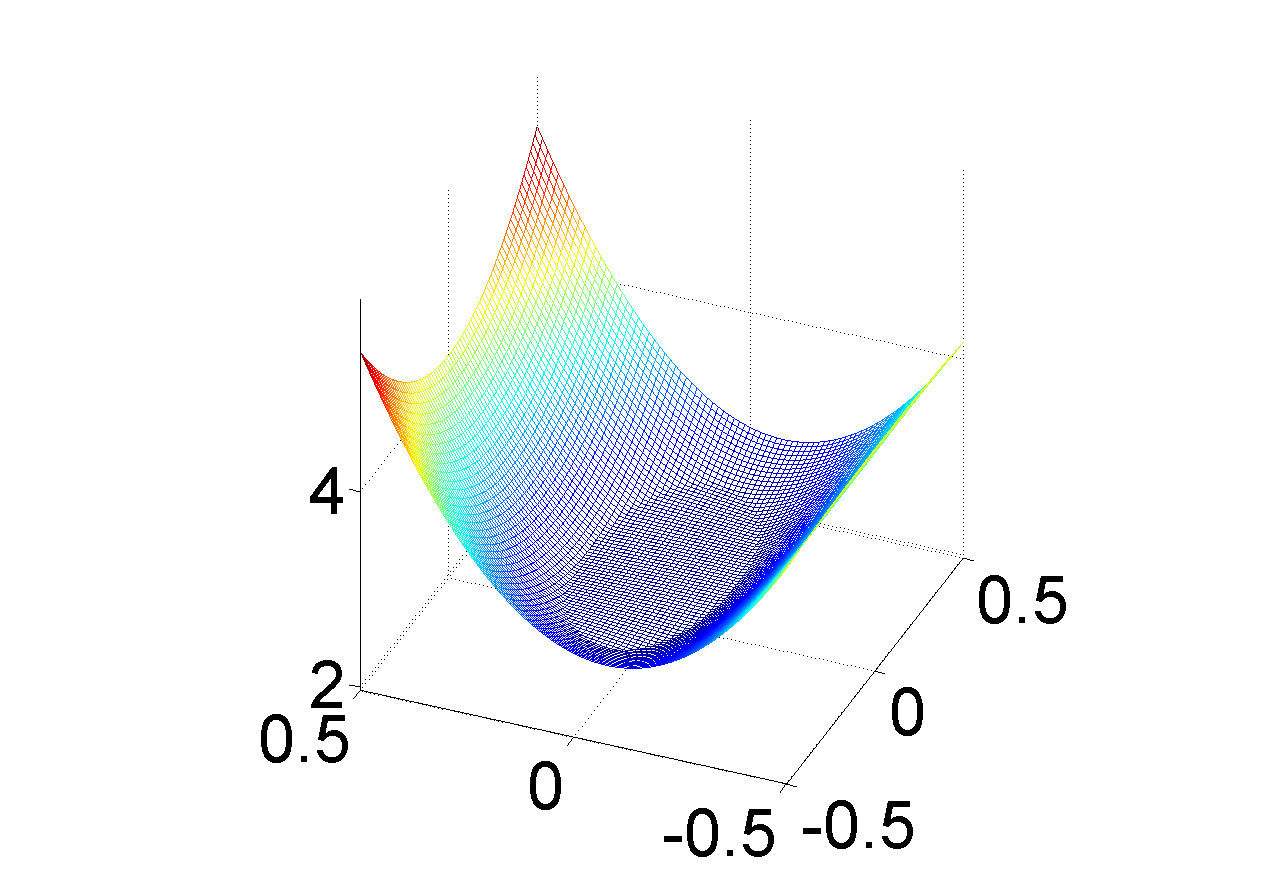}}
\subfigure[Error $u_2$]{
\label{Fig.sub.4}
\includegraphics[width=0.45\textwidth]{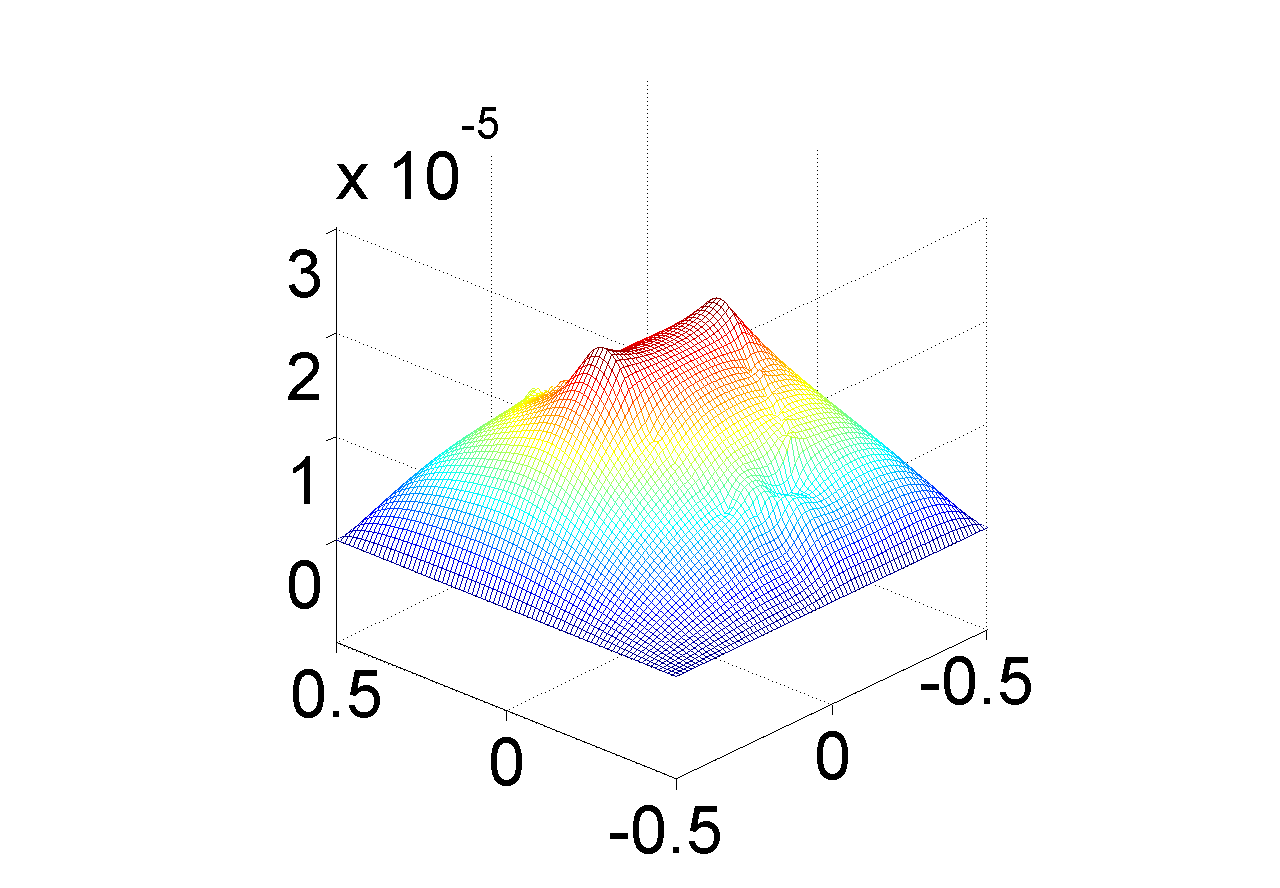}}
\caption{Numerical results for Example 1a on mesh $80\times 80$.}
\label{Fig.lable1}
\end{figure}

\paragraph{Example ~1b.}
We first increase the contrast between Poisson's ratios  in different subdomain to test the robustness of the present MIB scheme. To this end,
we keep shear modulus of the last case unchanged, while set the Poisson's ratio to be
$$
\nu=
\left\{\begin{array}{ll}
\nu^+=0.00024, &\ \  \mbox{in}\ \ \Omega^+,\\
\nu^-=0.24, &\ \  \mbox{in}\ \ \Omega^-.
\end{array}\right.
$$

Table \ref{ex1_large_por} displays the grid refinement analysis of the present numerical scheme.
 Obviously, the large contrast in the Poisson's ratios does not reduce the accuracy and convergent order of the present MIB method.
\begin{table}[!ht]
\centering
\caption{Numerical error and order for Example 1b with large contrast in Poisson's ratios.}
\begin{tabular}{lllllllll}
\cline{1-9}
$n_x\times n_y$ &$L_\infty(u_1)$&Order &$L_2(u_1)$&Order & $L_\infty(u_2)$&Order &$L_2(u_2)$ & Order  \\
\hline
$20\times 20$ &$3.84\times 10^{-4}$     &           &$2.19\times 10^{-4}$      &   &$1.89\times 10^{-4}$     &                 &$1.04\times 10^{-4}$      &\\
$40\times 40$ &$1.06\times 10^{-4}$     &1.86       &$5.70\times 10^{-5}$      &1.94  &$8.10\times 10^{-5}$     &1.24          &$3.75\times 10^{-5}$    &1.48\\
$80\times 80$ &$2.34\times 10^{-5}$     &2.18       &$1.24\times 10^{-5}$      &2.20  &$1.85\times 10^{-5}$     &2.13          &$8.48\times 10^{-6}$    &2.15\\
$160\times 160$ &$5.66\times 10^{-6}$     &2.05       &$3.01\times 10^{-6}$      &2.04  &$4.39\times 10^{-6}$     &2.18          &$2.08\times 10^{-6}$     &2.03\\
$320\times 320$ &$1.33\times 10^{-6}$     &2.09       &$7.40\times 10^{-7}$      &2.02  &$9.70\times 10^{-7}$     &2.24         &$4.80\times 10^{-7}$      &2.12\\
\hline
\end{tabular}
\label{ex1_large_por}
\end{table}

\paragraph{Example ~1c.}
We also analyze the performance of the proposed MIB method for large shear modulus contrast.
The Poisson's ratio is the same as that of Example 1a,  while  the shear modulus is set to
$$
\mu=
\left\{\begin{array}{ll}
\mu^+=2000, &\ \  \mbox{in}\ \ \Omega^+,\\
\mu^-=2000000, &\ \  \mbox{in}\ \ \Omega^-.
\end{array}\right.
$$
The grid refinement analysis for numerical error and order is shown in Table \ref{ex1_large_shr}. It is seen that the proposed MIB method
is very robust with respect to the large  contrast shear moduli.
\begin{table}[!ht]
\centering
\caption{Numerical error and order for Example 1c with large  shear modulus contrast.}
\begin{tabular}{lllllllll}
\cline{1-9}
$n_x\times n_y$ &$L_\infty(u_1)$&Order &$L_2(u_1)$&Order & $L_\infty(u_2)$&Order &$L_2(u_2)$ & Order  \\
\hline
$20\times 20$ &$2.87\times 10^{-4}$     &           &$1.55\times 10^{-4}$      &   &$2.51\times 10^{-4}$     &                 &$8.83\times 10^{-5}$      &\\
$40\times 40$ &$8.29\times 10^{-4}$     &1.79       &$4.12\times 10^{-5}$      &1.92  &$7.12\times 10^{-5}$     &1.82          &$3.06\times 10^{-5}$    &1.63\\
$80\times 80$ &$2.01\times 10^{-5}$     &2.04       &$9.50\times 10^{-6}$      &2.12  &$1.65\times 10^{-5}$     &2.11          &$7.38\times 10^{-6}$    &2.05\\
$160\times 160$ &$6.92\times 10^{-6}$     &1.55       &$2.89\times 10^{-6}$      &1.72  &$4.59\times 10^{-6}$     &1.85          &$1.83\times 10^{-6}$     &2.01\\
$320\times 320$ &$1.72\times 10^{-6}$     &2.01       &$6.10\times 10^{-7}$      &2.24  &$1.04\times 10^{-7}$     &2.14         &$4.20\times 10^{-7}$      &2.12\\
\hline
\end{tabular}
\label{ex1_large_shr}
\end{table}

\paragraph{Example 2.}
In this example, the computational domain is set to $\Omega=[-1, 1]\times [-1, 1]$.  The interface  is
 defined as circle $x^2+y^2=0.25$.

The Dirichlet boundary condition and the interface conditions are determined from the following designed exact solution
$$
u_1(x, y)=
\left\{\begin{array}{ll}
-r^2, &\ \  \mbox{in}\ \ \Omega^+,\\
-(r^4+c_0\log{(2r)})/10-r_0^2+(r_0^4+c_0\log{(2r_0)})/10, &\ \  \mbox{in}\ \ \Omega^-.
\end{array}\right.
$$
and
$$
u_2(x, y)=
\left\{\begin{array}{ll}
\log{(1+x^2+3y^2)}+\sin(xy), &\ \  \mbox{in}\ \ \Omega^+,\\
\log{(1+x^2+3y^2)}+\sin(xy)-4r^2+4r_0^2, &\ \  \mbox{in}\ \ \Omega^-.
\end{array}\right.
$$
where $r_0=0.5, c_0=-0.1$.

\paragraph{Example 2a.}
 We first consider  a standard test case used in the field \cite{YangXZ:2003}.
The Poisson's ratio and the shear modulus are, respectively,
$$
\nu=
\left\{\begin{array}{ll}
\nu^+=0.20, &\ \  \mbox{in}\ \ \Omega^+,\\
\nu^-=0.24, &\ \  \mbox{in}\ \ \Omega^-.
\end{array}\right.
$$
and
$$
\mu=
\left\{\begin{array}{ll}
\mu^+=2500000, &\ \  \mbox{in}\ \ \Omega^+,\\
\mu^-=3000000, &\ \  \mbox{in}\ \ \Omega^-.
\end{array}\right.
$$

\begin{table}[!ht]
\centering
\caption{Numerical error and order for Example 2a.}
\begin{tabular}{lllllllll}
\cline{1-9}
$n_x\times n_y$ &$L_\infty(u_1)$&Order &$L_2(u_1)$&Order & $L_\infty(u_2)$&Order &$L_2(u_2)$ & Order  \\
\hline
$20\times 20$ &$3.10\times 10^{-3}$     &           &$1.23\times 10^{-3}$      &      &$8.97\times 10^{-2}$     &              &$3.51\times 10^{-3}$      &\\
$40\times 40$ &$8.62\times 10^{-4}$     &1.85       &$2.98\times 10^{-4}$      &2.04  &$1.04\times 10^{-3}$     &3.09          &$3.68\times 10^{-4}$     &3.10\\
$80\times 80$ &$1.90\times 10^{-4}$     &2.18       &$7.61\times 10^{-5}$      &1.97  &$2.27\times 10^{-4}$     &2.20          &$8.27\times 10^{-4}$     &2.15\\
$160\times 160$ &$4.39\times 10^{-5}$     &2.11       &$1.91\times 10^{-5}$      &1.95  &$6.04\times 10^{-5}$     &1.91          &$2.01\times 10^{-5}$     &2.03\\
$320\times 320$ &$1.26\times 10^{-5}$     &1.80       &$4.96\times 10^{-6}$      &1.95  &$2.04\times 10^{-5}$     &1.57          &$7.10\times 10^{-6}$     &1.53\\
\hline
\end{tabular}
\label{ex2_mib}
\end{table}

Table \ref{ex2_mib} gives the grid refinement analysis. Figure \ref{Fig.lable2} plots our numerical  results. Essentially, the designed order and accuracy are obtained.  The slight order reduction at the last mesh was due to slow varying nature of the solution over the computational domain. This reduction can be eliminated when some perturbation is introduced as shown in the next numerical example.

\begin{figure}[!ht]
\centering
\subfigure[Numerical solution $u_1$]{
\label{Fig.sub.21}
\includegraphics[width=0.45\textwidth]{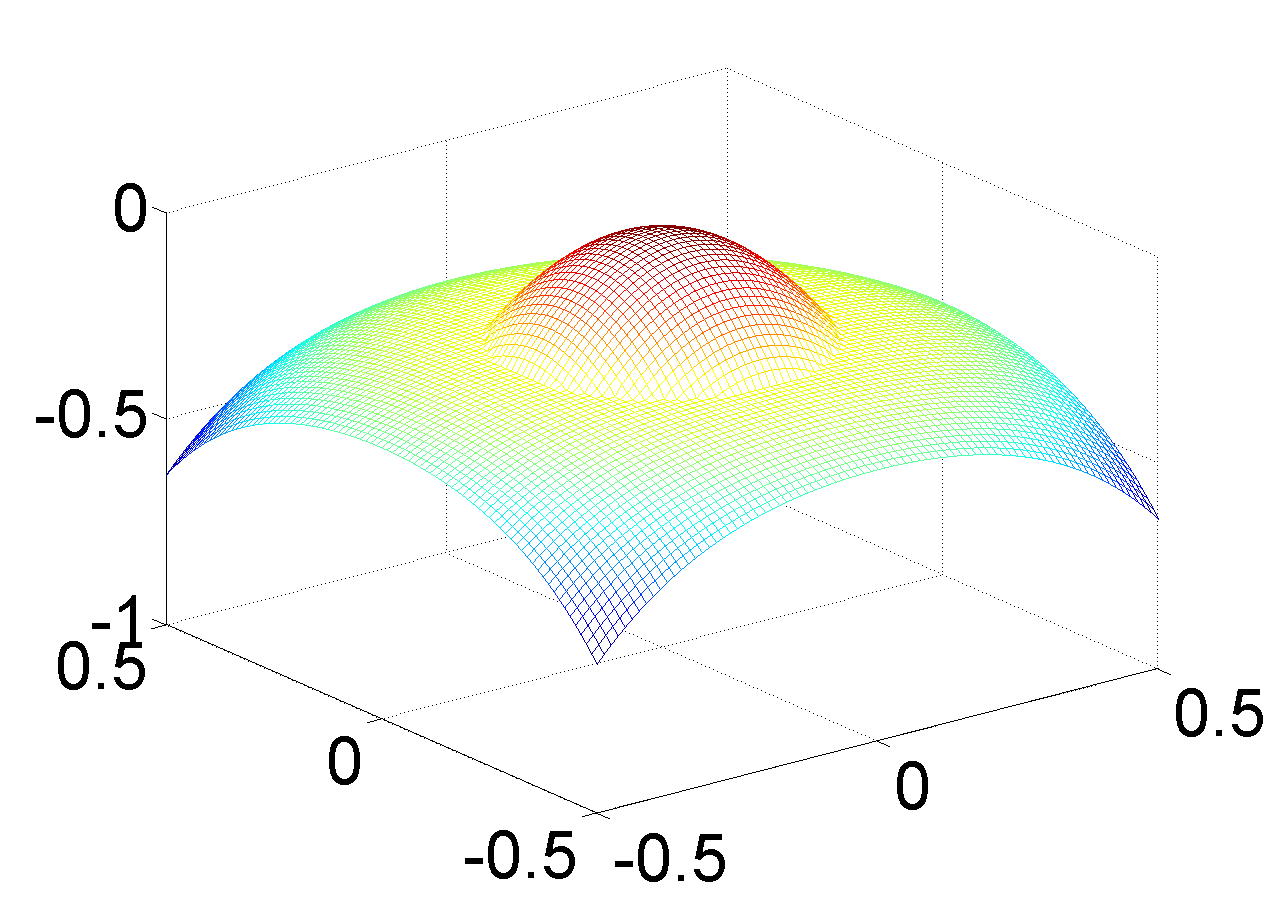}}
\subfigure[Error $u_1$]{
\label{Fig.sub.22}
\includegraphics[width=0.45\textwidth]{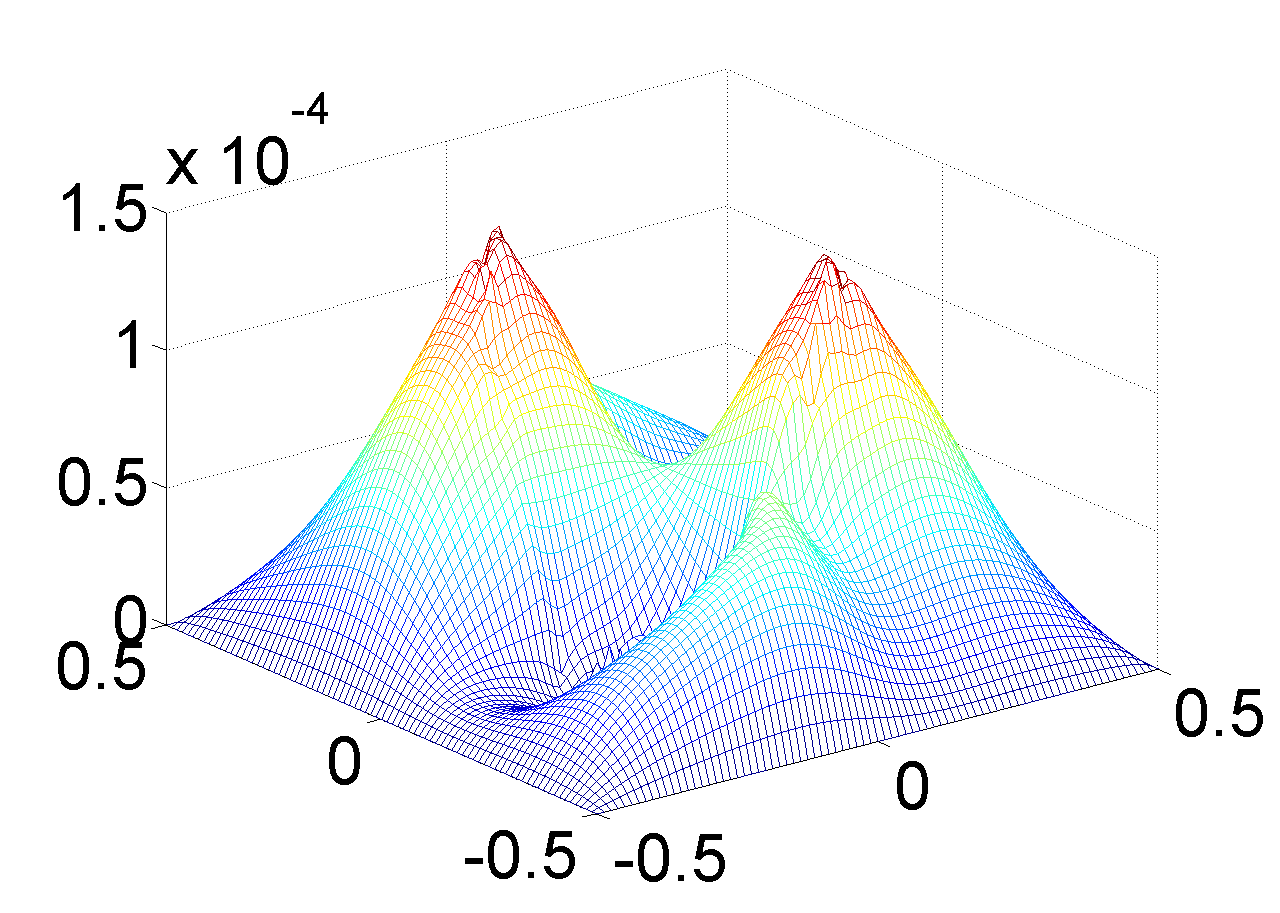}}
\subfigure[Numerical solution $u_2$]{
\label{Fig.sub.23}
\includegraphics[width=0.45\textwidth]{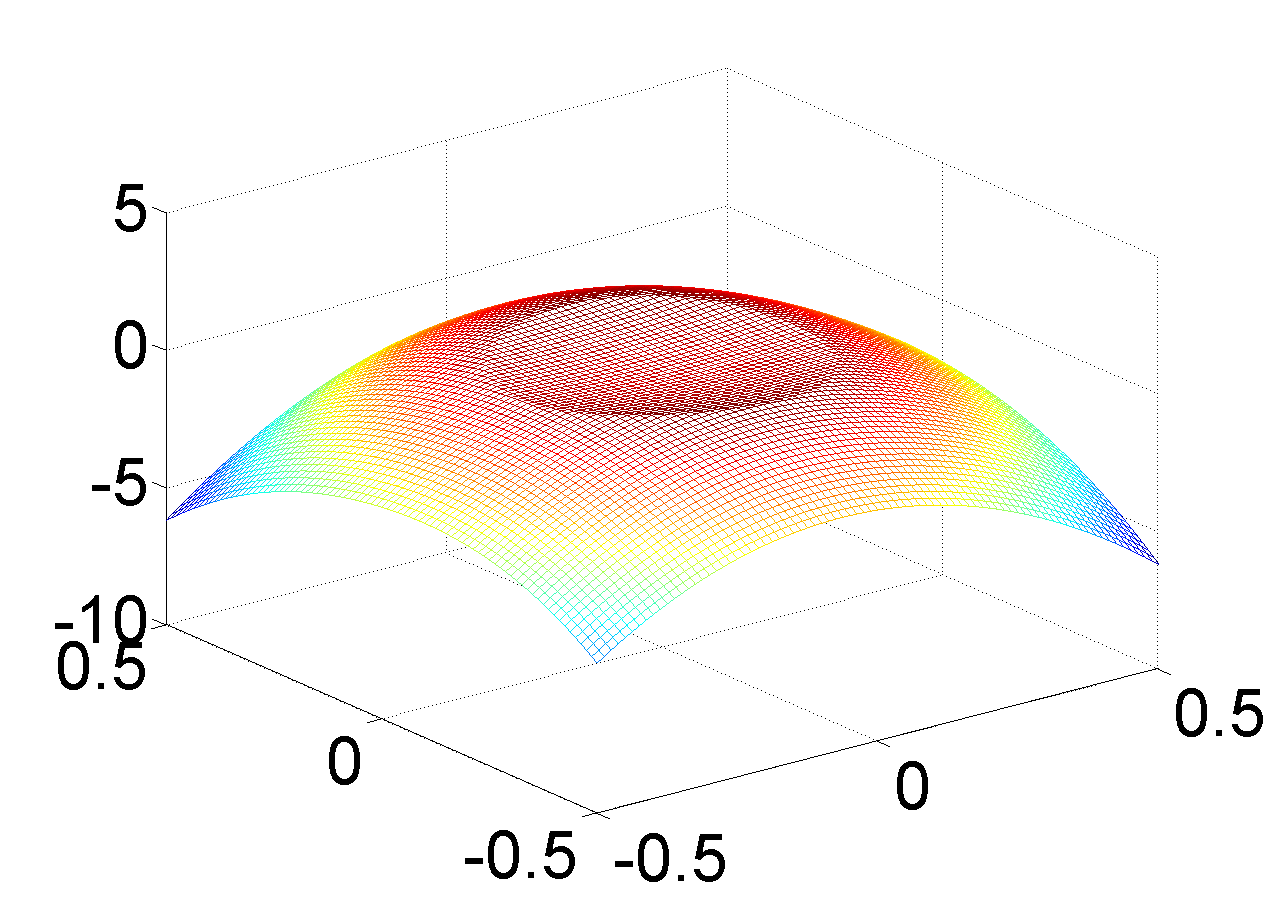}}
\subfigure[Error $u_2$]{
\label{Fig.sub.24}
\includegraphics[width=0.45\textwidth]{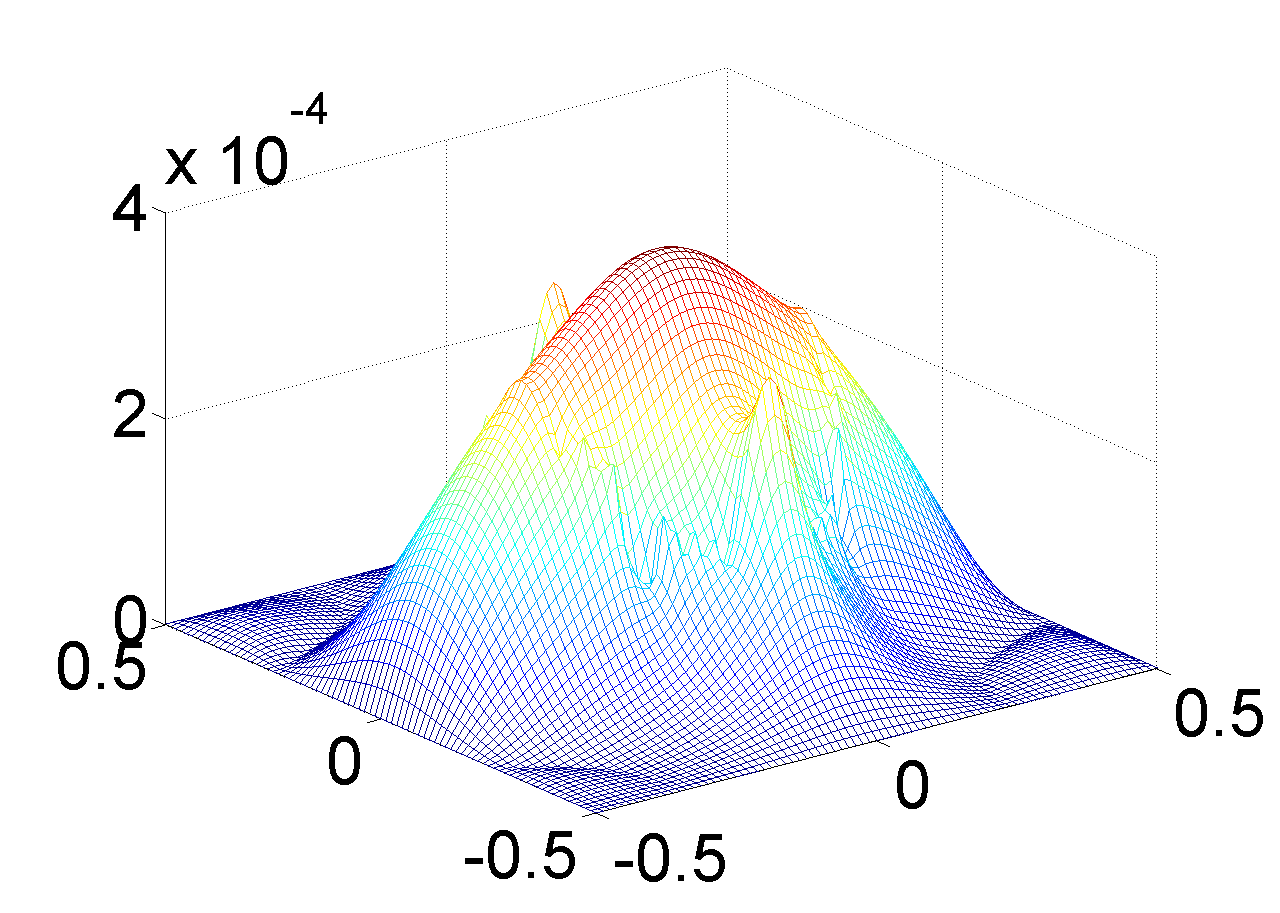}}
\caption{Numerical results for Example 2a on mesh $80\times 80$.}
\label{Fig.lable2}
\end{figure}

\paragraph{Example 2b.}
It is important to know whether the proposed method is robust for large contrast in Poisson's ratio and shear modulus.
We consider a change in the Poisson's ratio while keep the shear modulus given in Example 2a
$$
\nu=
\left\{\begin{array}{ll}
\nu^+=0.00024, &\ \  \mbox{in}\ \ \Omega^+,\\
\nu^-=0.24, &\ \  \mbox{in}\ \ \Omega^-.
\end{array}\right.
$$
Table \ref{ex2_large_por} gives the grid refinement analysis of the numerical scheme. It is seen that the MIB method is robust for large contrast in  Poisson's ratios.

\begin{table}[!ht]
\caption{Numerical error and order for Example 2b with large  Poisson's ratio contrast.}
\centering
\begin{tabular}{lllllllll}
\cline{1-9}
$n_x\times n_y$ &$L_\infty(u_1)$&Order &$L_2(u_1)$&Order &$L_\infty(u_2)$&Order &$L_2(u_2)$ & Order  \\
\hline
$20\times 20$ &$3.39\times 10^{-3}$     &           &$1.35\times 10^{-3}$      &      &$9.53\times 10^{-2}$     &              &$3.78\times 10^{-3}$      &\\
$40\times 40$ &$8.92\times 10^{-4}$     &1.92       &$3.04\times 10^{-4}$      &2.15  &$9.37\times 10^{-3}$     &3.35          &$2.91\times 10^{-4}$    &3.70\\
$80\times 80$ &$1.99\times 10^{-4}$     &2.16       &$7.61\times 10^{-5}$      &2.00  &$2.22\times 10^{-4}$     &2.08          &$7.59\times 10^{-4}$    &1.94\\
$160\times 160$ &$4.79\times 10^{-5}$     &2.00       &$1.96\times 10^{-5}$      &1.96  &$6.95\times 10^{-5}$     &1.68          &$2.36\times 10^{-5}$     &1.69\\
$320\times 320$ &$1.34\times 10^{-5}$     &1.89       &$4.88\times 10^{-6}$      &2.01  &$1.73\times 10^{-5}$     &2.01         &$5.93\times 10^{-6}$      &1.99\\
\hline
\end{tabular}\label{ex2_large_por}
\end{table}

\paragraph{Example 2c.}
Finally, we consider a large shear modulus contrast. The Poisson's ratio is the same as Example 2a, while shear modulus is given by
$$
\mu=
\left\{\begin{array}{ll}
\mu^+=3000, &\ \  \mbox{in}\ \ \Omega^+,\\
\mu^-=3000000, &\ \  \mbox{in}\ \ \Omega^-.
\end{array}\right.
$$
The grid refinement analysis for the error and order is shown in Table  \ref{ex2_large_shr}. The second order convergence is maintained.
\begin{table}[!ht]
\centering
\caption{Numerical error and order for Example 2c with large shear modulus contrast.}
\begin{tabular}{lllllllll}
\cline{1-9}
$n_x\times n_y$ &$L_\infty(u_1)$&Order &$L_2(u_1)$&Order &$L_\infty(u_2)$&Order &$L_2(u_2)$ & Order  \\
\hline
$20\times 20$ &$5.99\times 10^{-3}$     &           &$1.60\times 10^{-3}$      &      &$9.90\times 10^{-3}$     &              &$3.21\times 10^{-3}$      &\\
$40\times 40$ &$1.45\times 10^{-3}$     &2.05       &$4.77\times 10^{-4}$      &1.75  &$2.04\time  10^{-3}$     &2.28          &$4.37\times 10^{-4}$     &2.88\\
$80\times 80$ &$3.66\times 10^{-4}$     &1.97       &$1.42\times 10^{-5}$      &1.75  &$5.41\times 10^{-4}$     &1.91          &$1.12\times 10^{-4}$     &1.96\\
$160\times 160$&$1.04\times 10^{-4}$     &1.99       &$2.83\times 10^{-5}$      &2.32  &$1.48\times 10^{-5}$     &1.87          &$4.26\times 10^{-5}$     &1.40\\
$320\times 320$&$2.50\times 10^{-5}$     &2.06       &$7.61\times 10^{-6}$      &1.89  &$3.15\times 10^{-5}$     &2.23         &$5.40\times 10^{-6}$      &2.98\\
\hline
\end{tabular}\label{ex2_large_shr}
\end{table}

\paragraph{Example 3.}
In this case, we consider a more complex interface to demonstrate the performance of the MIB method. The interface is of a flower shape and is defined in the polar coordinate $$r=0.5+\frac{\sin5\theta}{7}.$$
We set the computational domain $\Omega=[-1, 1]\times [-1, 1]$.
The Dirichlet boundary condition and the interface conditions are determined from the following  exact solution
$$
u_1(x, y)=
\left\{\begin{array}{ll}
\exp{(-3.5^2(x^2+y^2)^5)}, &\ \  \mbox{in}\ \ \Omega^+,\\
\exp{(-(7(x^2+y^2)^3-5x^4y+10x^2y^3-y^5)^2)}, &\ \  \mbox{in}\ \ \Omega^-.
\end{array}\right.
$$
and
$$
u_2(x, y)=
\left\{\begin{array}{ll}
\exp{(-3.5^2(x^2+y^2)^5)}+xy, &\ \  \mbox{in}\ \ \Omega^+,\\
\exp{(-(7(x^2+y^2)^3-5x^4y+10x^2y^3-y^5)^2)}+xy, &\ \  \mbox{in}\ \ \Omega^-.
\end{array}\right.
$$
We consider two cases for this problem below.

\paragraph{Example 3a.} We first set the Poisson's ratio and the shear as the follows
$$
\nu=
\left\{\begin{array}{ll}
\nu^+=0.20, &\ \  \mbox{in}\ \ \Omega^+,\\
\nu^-=0.24, &\ \  \mbox{in}\ \ \Omega^-.
\end{array}\right.
$$
and
$$
\mu=
\left\{\begin{array}{ll}
\mu^+=1500000, &\ \  \mbox{in}\ \ \Omega^+,\\
\mu^-=2000000, &\ \  \mbox{in}\ \ \Omega^-.
\end{array}\right.
$$

\begin{table}[!ht]
\centering
\caption{Numerical error and order for Example 3a.}
\begin{tabular}{lllllllll}
\cline{1-9}
$n_x\times n_y$ &$L_\infty(u_1)$&Order &$L_2(u_1)$&Order &$L_\infty(u_2)$&Order &$L_2(u_2)$ & Order  \\
\hline
$20\times 20$ &$2.83\times 10^{-2}$     &           &$8.92\times 10^{-3}$      &   &$2.66\times 10^{-2}$     &                 &$5.86\times 10^{-3}$      &\\
$40\times 40$ &$5.10\times 10^{-3}$     &2.47       &$1.71\times 10^{-3}$      &2.38  &$5.58\times 10^{-3}$     &2.25          &$1.76\times 10^{-3}$    &1.74\\
$80\times 80$ &$2.08\times 10^{-3}$     &1.30       &$6.54\times 10^{-4}$      &1.39  &$2.29\times 10^{-3}$     &1.29          &$6.83\times 10^{-4}$    &1.37\\
$160\times 160$ &$5.09\times 10^{-4}$     &2.00       &$1.72\times 10^{-4}$      &1.93  &$4.99\times 10^{-4}$     &2.20          &$1.66\times 10^{-4}$     &2.04\\
$320\times 320$ &$1.26\times 10^{-4}$     &2.01       &$4.21\times 10^{-5}$      &2.03  &$1.08\times 10^{-4}$     &2.18         &$3.62\times 10^{-5}$      &2.20\\
\hline
\end{tabular}
\label{flower_err}
\end{table}

The grid refinement analysis is listed in Table \ref{flower_err}. Figure \ref{Fig.lable3} demonstrates our numerical results.
Clearly, the second order convergence is obtained for this irregular interface.

\begin{figure}[!ht]
\centering
\subfigure[Numerical Solution $u_1$]{
\label{Fig.sub.31}
\includegraphics[width=0.45\textwidth]{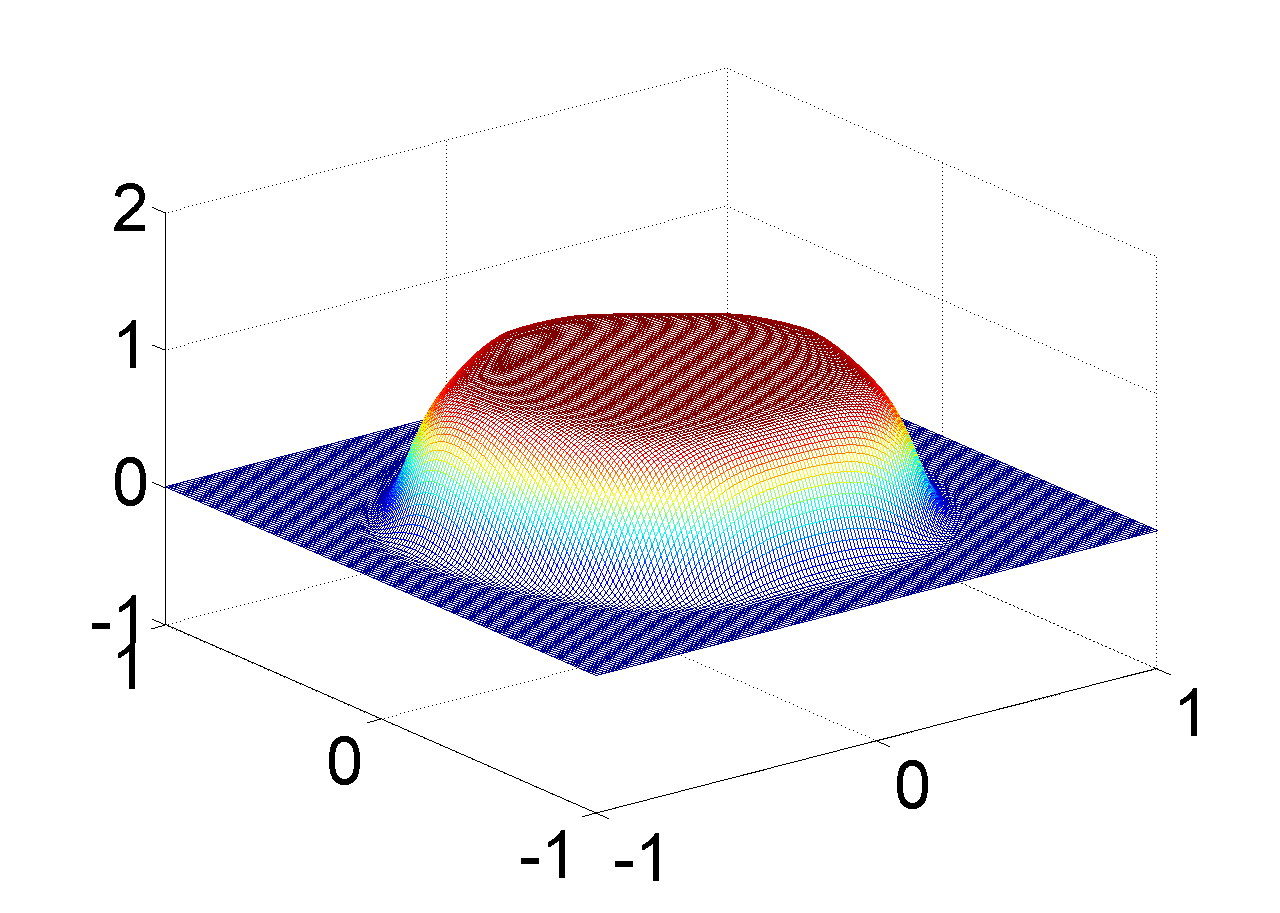}}
\subfigure[Error $u_1$]{
\label{Fig.sub.32}
\includegraphics[width=0.45\textwidth]{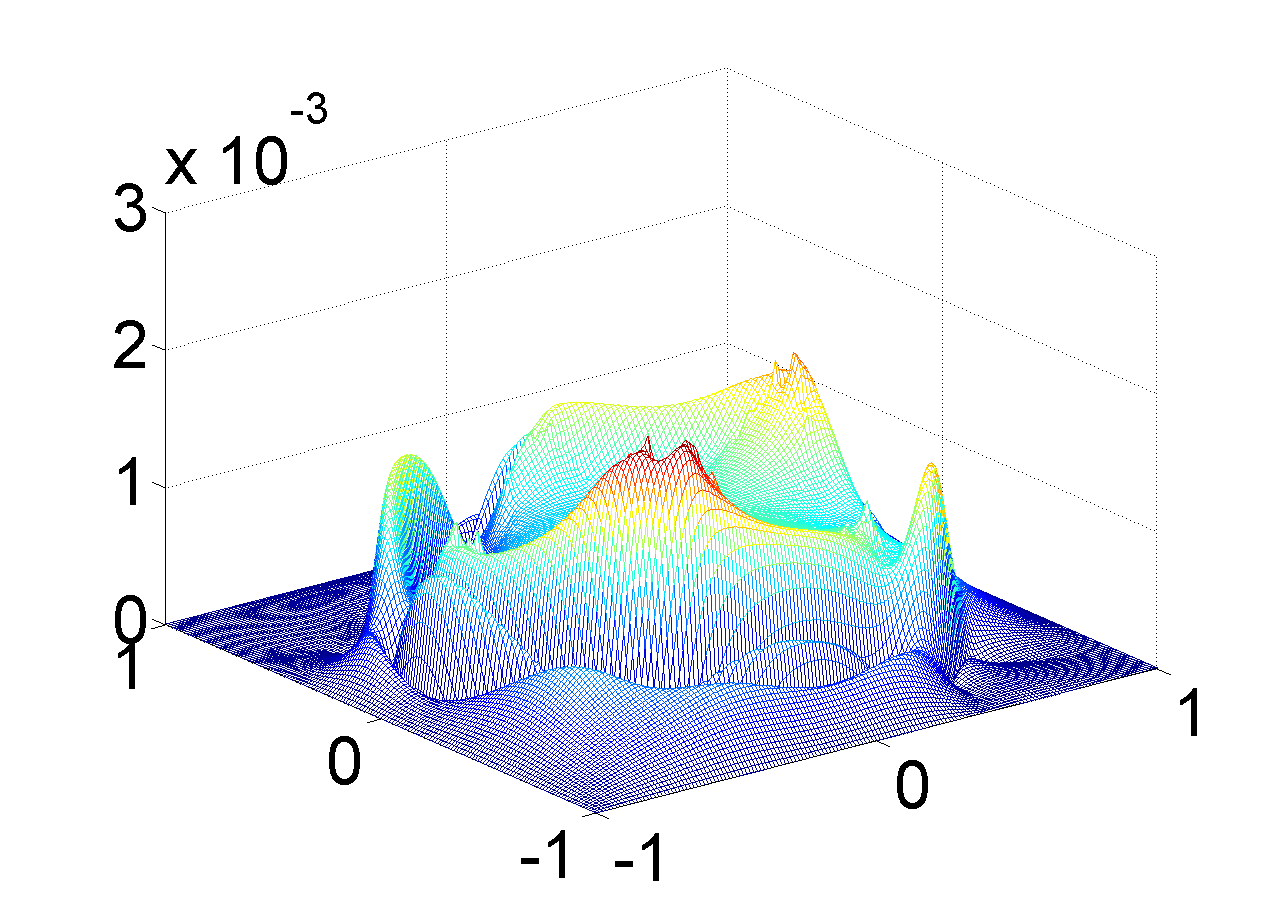}}
\subfigure[Numerical Solution $u_2$]{
\label{Fig.sub.33}
\includegraphics[width=0.45\textwidth]{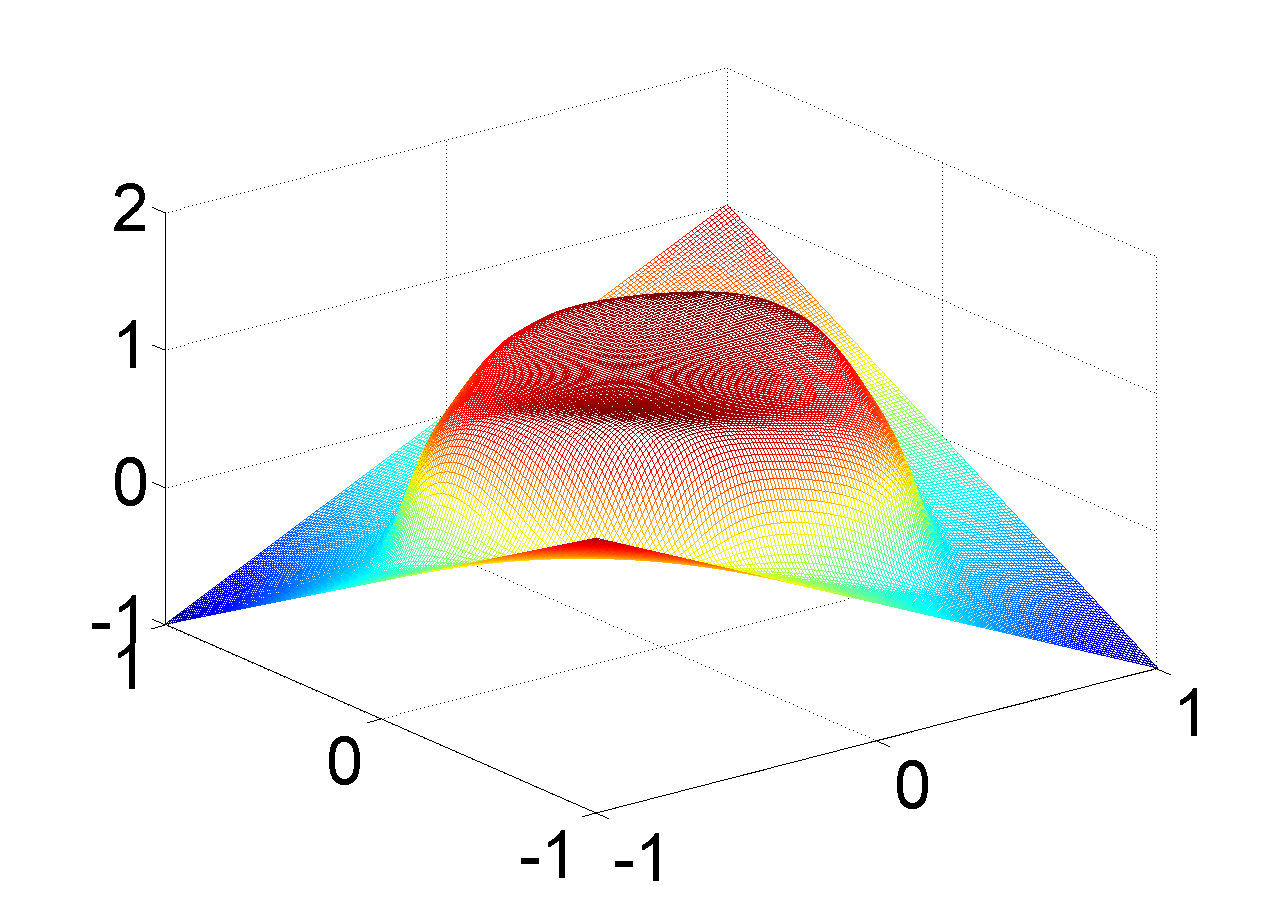}}
\subfigure[Error $u_2$]{
\label{Fig.sub.34}
\includegraphics[width=0.45\textwidth]{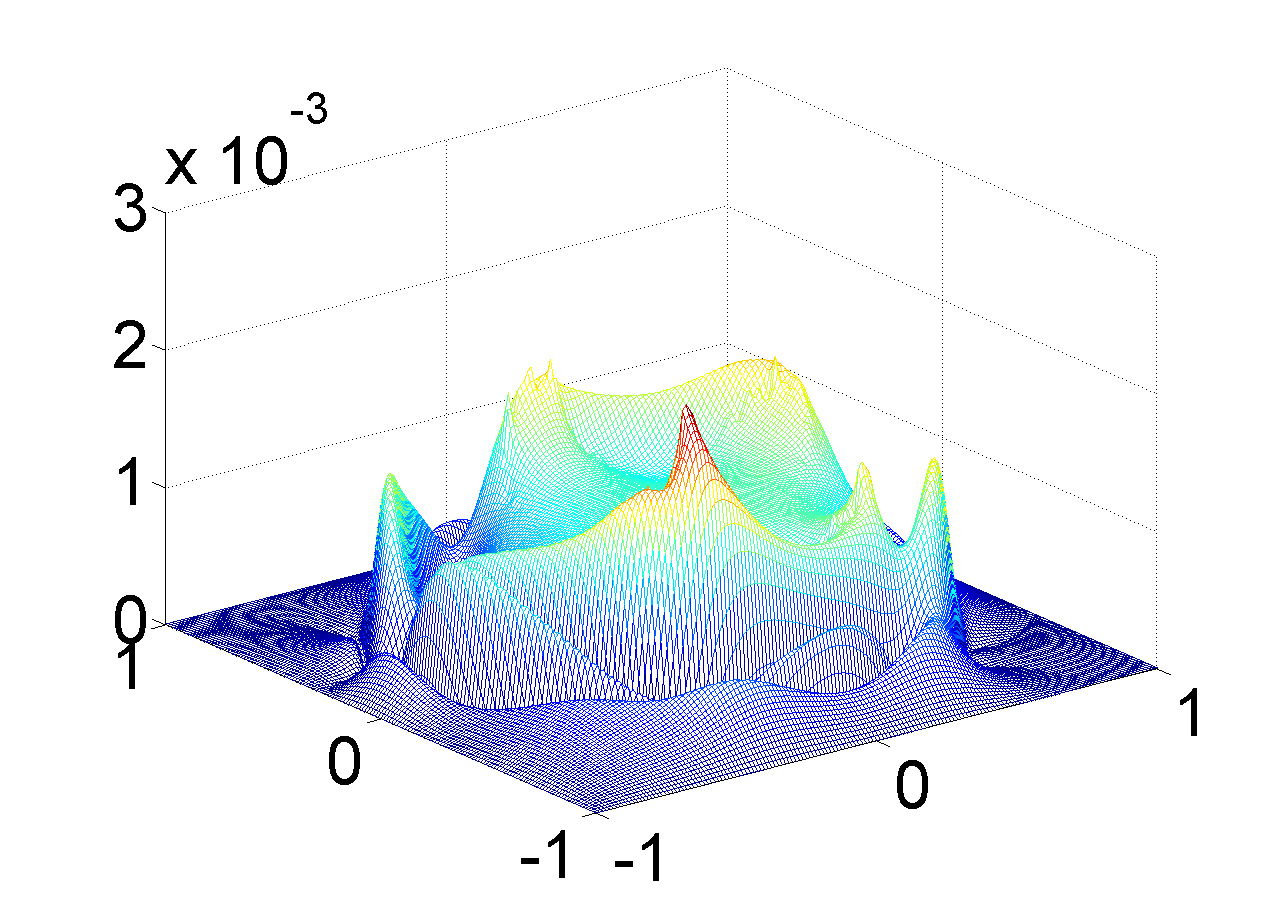}}
\caption{Numerical results for Example 3a on mesh $80\times 80$.}
\label{Fig.lable3}
\end{figure}

\paragraph{Example 3b.} We next test the robustness of the present method for large Poisson's ratios.
We keep the shear modulus the same as that in Example 3a, while change Poisson's ratios to
$$
\nu=
\left\{\begin{array}{ll}
\nu^+=0.00024, &\ \  \mbox{in}\ \ \Omega^+,\\
\nu^-=0.24, &\ \  \mbox{in}\ \ \Omega^-.
\end{array}\right.
$$

Table \ref{ex3_large_por} presents the grid refinement analysis of the numerical scheme.
Again, we see the designed second order convergence in both $L_\infty$ and $L_2$ errors.

\begin{table}[!ht]
\centering
\caption{Numerical error and order for Example 3b with large contrast in Poisson's ratios.}
\begin{tabular}{lllllllll}
\cline{1-9}
$n_x\times n_y$ &$L_\infty(u_1)$&Order &$L_2(u_1)$&Order & $L_\infty(u_2)$&Order &$L_2(u_2)$ & Order  \\
\hline
$20\times 20$ &$3.22\times 10^{-2}$     &           &$1.04\times 10^{-2}$      &   &$2.46\times 10^{-2}$     &                 &$7.17\times 10^{-3}$      &\\
$40\times 40$ &$6.48\times 10^{-3}$     &2.31       &$2.15\times 10^{-3}$      &2.27  &$6.22\times 10^{-3}$     &1.98          &$2.20\times 10^{-3}$    &1.71\\
$80\times 80$ &$2.48\times 10^{-3}$     &1.39       &$7.70\times 10^{-4}$      &1.48  &$2.60\times 10^{-3}$     &1.26          &$7.92\times 10^{-4}$    &1.48\\
$160\times 160$ &$6.01\times 10^{-4}$     &2.04       &$1.97\times 10^{-4}$      &1.97  &$5.75\times 10^{-4}$     &2.18          &$1.90\times 10^{-4}$     &2.06\\
$320\times 320$ &$1.40\times 10^{-4}$     &2.10       &$4.78\times 10^{-5}$      &2.04  &$1.26\times 10^{-4}$     &2.19         &$4.27\times 10^{-5}$      &2.15\\
\hline
\end{tabular}\label{ex3_large_por}
\end{table}

\subsubsection{Strong discontinuity}
\begin{remark}
The above numerical examples verify that the proposed MIB scheme is essentially of second order convergence for the weak discontinuity case. Now we turn to verify the efficiency and robustness for the strong discontinuity scenario.
\end{remark}

\paragraph{Example 4.}
To further examine our method   for complicated  interface geometry, we reconsider the exact solution defined in Example 1, while  change the interface to the flower-like pattern as defined in Example 3.

\begin{table}[!ht]
\centering
\caption{Numerical error and order for Example 4.}
\begin{tabular}{lllllllll}
\cline{1-9}
$n_x\times n_y$ &$L_\infty(u_1)$&Order &$L_2(u_1)$&Order & $L_\infty(u_2)$&Order &$L_2(u_2)$ & Order  \\
\hline
$20\times 20$ &$1.38\times 10^{-3}$     &           &$6.42\times 10^{-4}$      &   &$1.55\times 10^{-3}$     &                 &$7.17\times 10^{-4}$      &\\
$40\times 40$ &$3.20\times 10^{-4}$     &2.11       &$1.67\times 10^{-4}$      &2.01  &$2.64\times 10^{-4}$     &2.55          &$1.32\times 10^{-4}$    &2.44\\
$80\times 80$ &$6.35\times 10^{-5}$     &2.33       &$3.06\times 10^{-5}$      &1.94  &$5.89\times 10^{-5}$     &2.16          &$3.00\times 10^{-5}$    &2.14\\
$160\times 160$ &$1.50\times 10^{-5}$     &2.08       &$7.15\times 10^{-6}$      &2.10  &$1.51\times 10^{-6}$     &1.96          &$7.69\times 10^{-6}$     &1.96\\
$320\times 320$ &$3.53\times 10^{-6}$     &2.09       &$1.82\times 10^{-6}$      &1.92  &$4.00\times 10^{-7}$     &1.92         &$2.08\times 10^{-6}$      &1.89\\
\hline
\end{tabular}
\label{dis_ex4}
\end{table}

Table (\ref{dis_ex4}) gives the grid refinement analysis. Results are also depicted in \ref{Fig.lable4}.
Essentially, the designed order of convergence is maintained.

\begin{figure}[!ht]
\centering
\subfigure[Numerical solution $u_1$]{
\label{Fig.sub.41}
\includegraphics[width=0.45\textwidth]{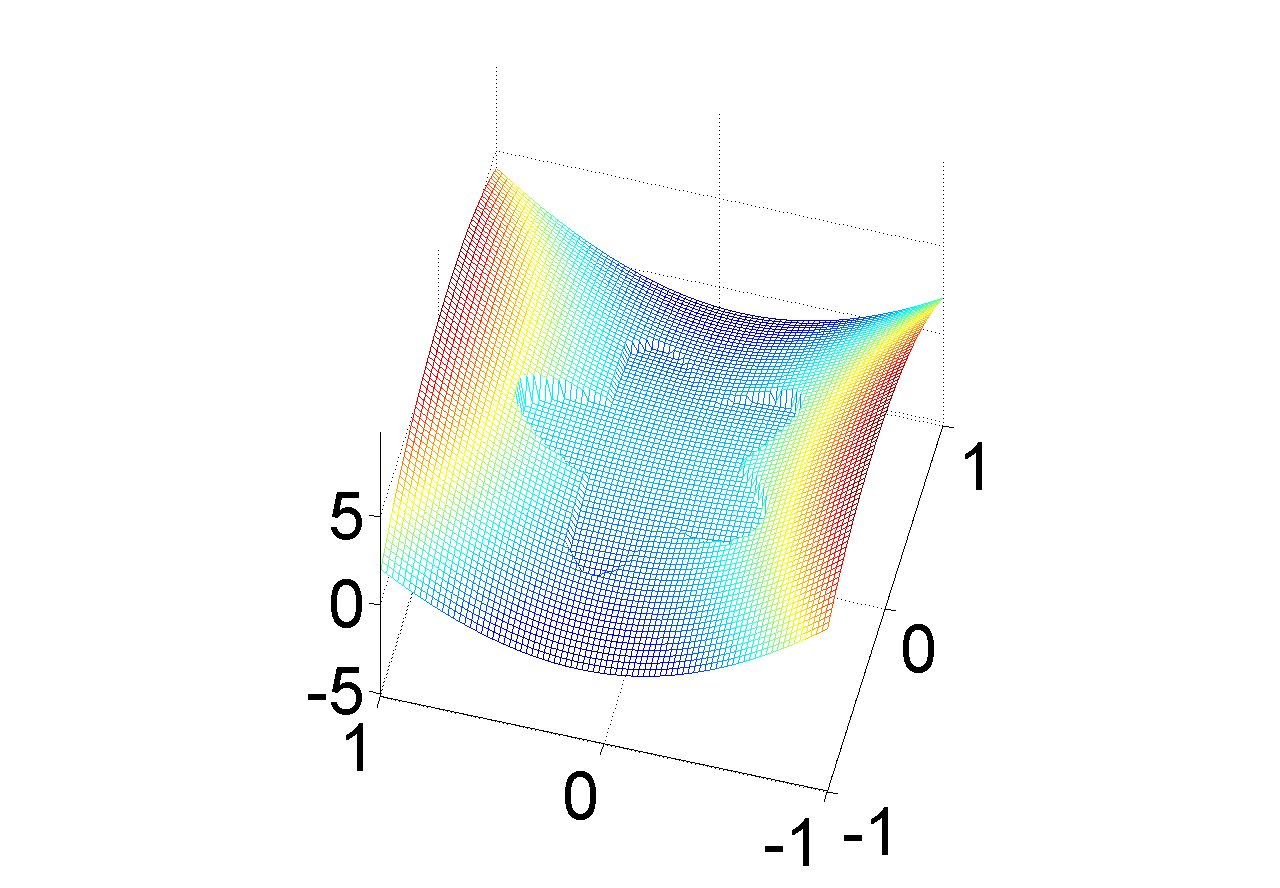}}
\subfigure[Error $u_1$]{
\label{Fig.sub.42}
\includegraphics[width=0.45\textwidth]{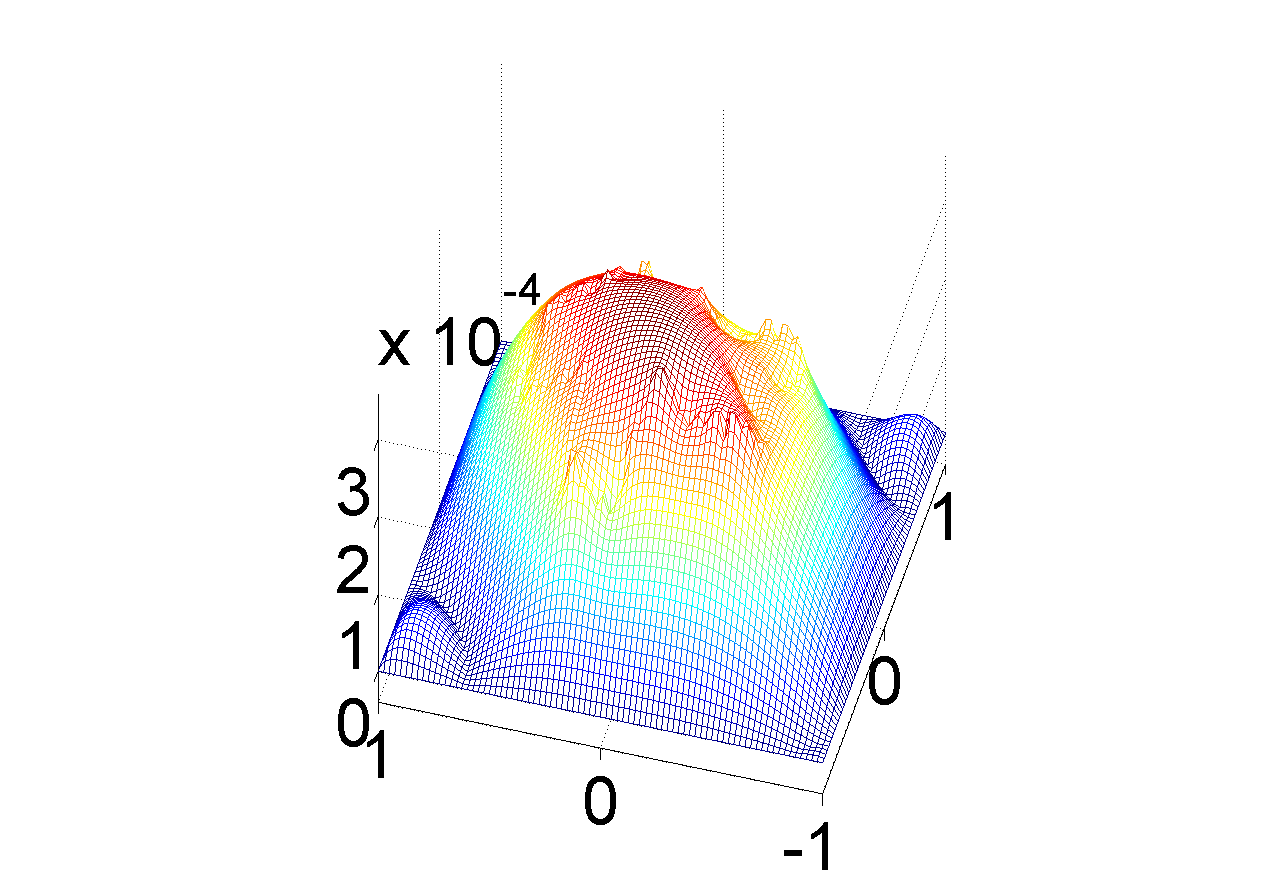}}
\subfigure[Numerical solution $u_2$]{
\label{Fig.sub.43}
\includegraphics[width=0.45\textwidth]{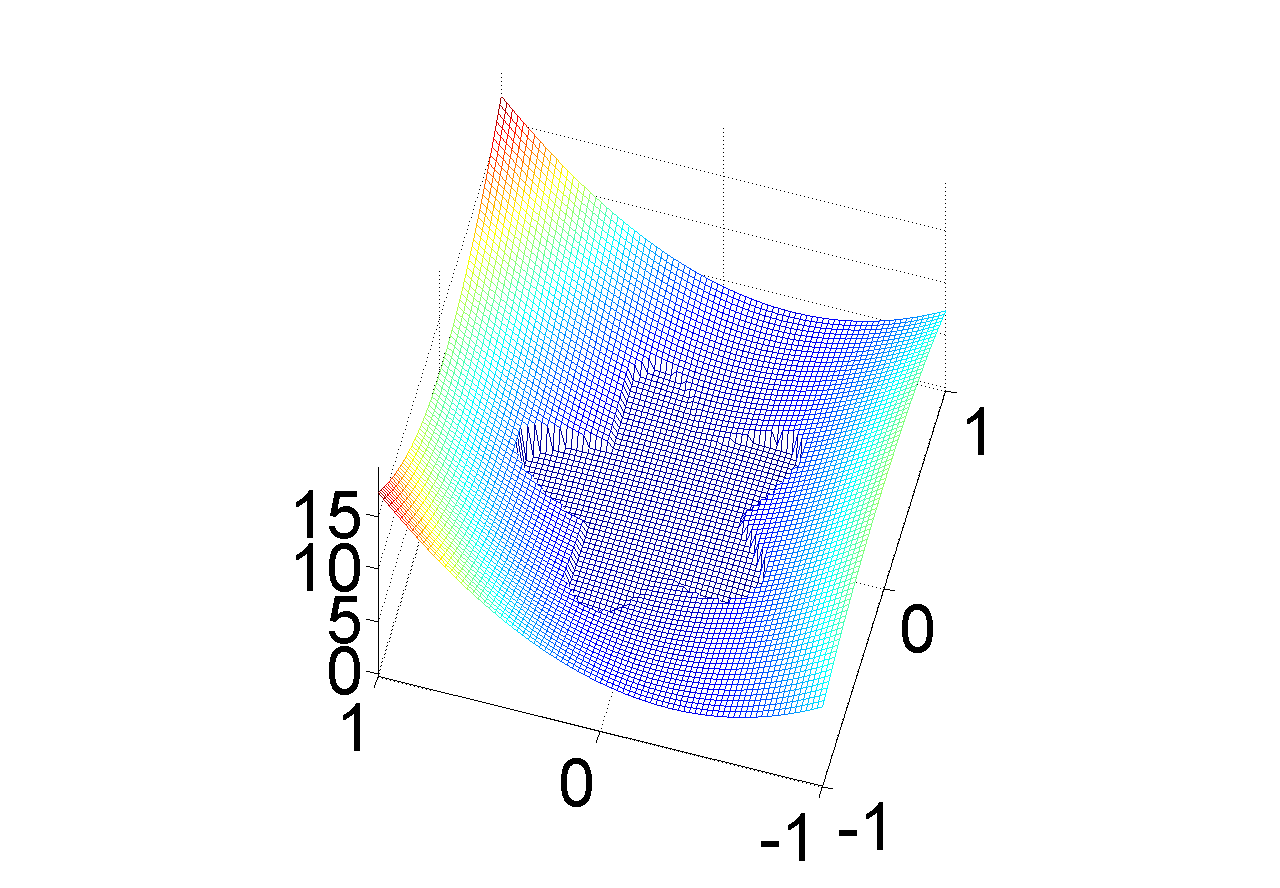}}
\subfigure[Error $u_2$]{
\label{Fig.sub.44}
\includegraphics[width=0.45\textwidth]{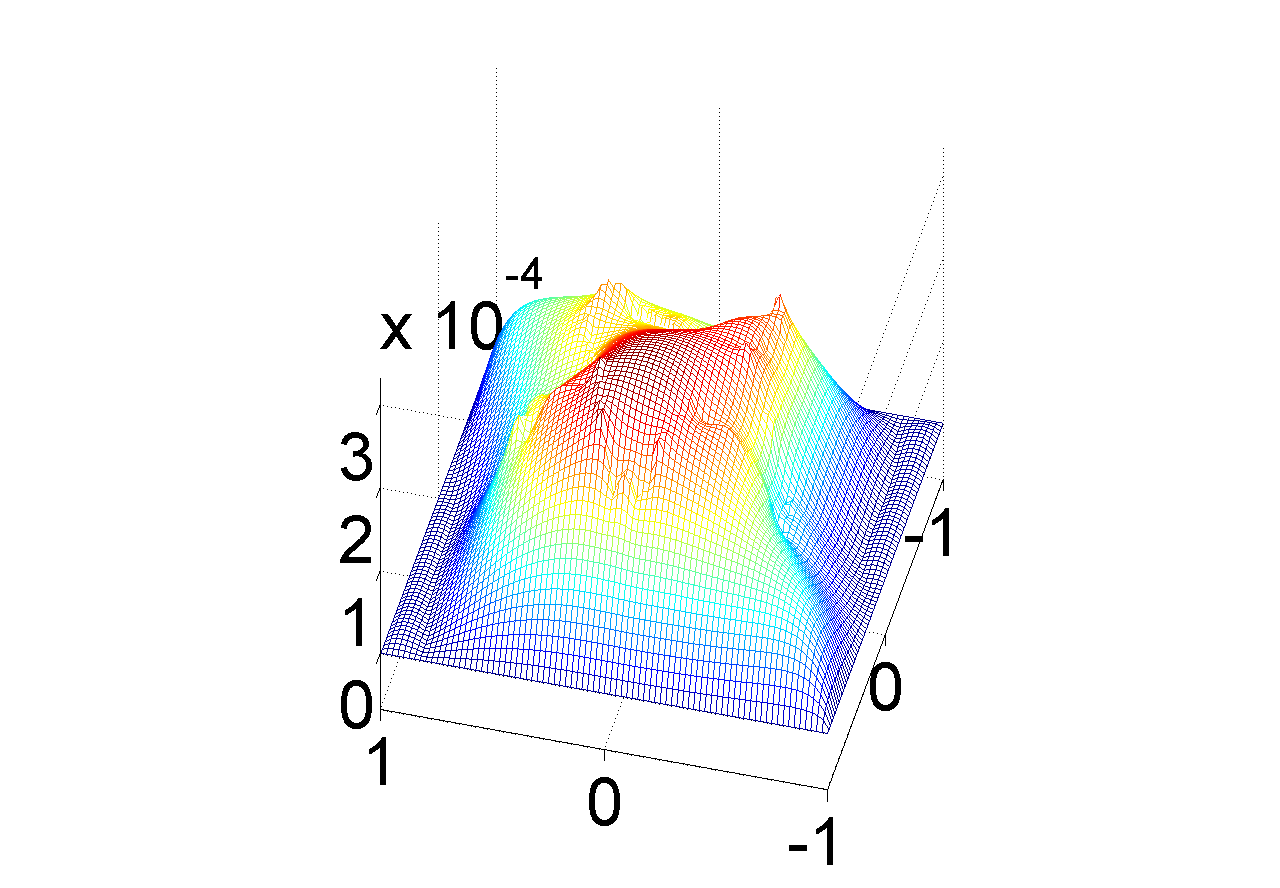}}
\caption{Numerical results for Example 4 on mesh $80\times 80$.}
\label{Fig.lable4}
\end{figure}

\paragraph{Example 5.}
To further test our method for strong discontinuity with complex geometric interface, we consider the exact solution given in Example 1 again, while now change the interface to the benchmark jigsaw-like pattern
$$
\left\{\begin{array}{ll}
x(\theta)=0.6\cos(\theta)-0.3\cos(3\theta), &\ \  \mbox{in}\ \ \Omega^+,\\
y(\theta)=1.5+0.7\sin(\theta)-0.07\sin(3\theta)+0.2\sin(7\theta), &\ \  \mbox{in}\ \ \Omega^-.
\end{array}\right.
$$
The computational domain is set to $\Omega=[-1, 1]\times [0, 3]$.
In this case, the interface geometry is very complex.

\begin{table}[!ht]
\centering
\caption{Numerical error and order for  Example 5.}
\begin{tabular}{lllllllll}
\cline{1-9}
$n_x\times n_y$ &$L_\infty(u_1)$&Order &$L_2(u_1)$&Order & $L_\infty(u_2)$&Order &$L_2(u_2)$ & Order  \\
\hline
$40\times 30$ &$9.07\times 10^{-3}$     &           &$2.53\times 10^{-3}$      &   &$8.42\times 10^{-3}$     &                 &$2.23\times 10^{-3}$      &\\
$80\times 60$ &$2.48\times 10^{-3}$     &1.87       &$6.87\times 10^{-4}$      &1.88  &$2.20\times 10^{-3}$     &1.94          &$6.11\times 10^{-4}$    &1.87\\
$160\times 120$ &$6.04\times 10^{-4}$     &2.04       &$1.71\times 10^{-4}$      &2.01  &$5.40\times 10^{-4}$     &2.03          &$1.60\times 10^{-4}$    &1.93\\
$320\times 240$&$1.58\times 10^{-4}$     &1.93       &$4.20\times 10^{-5}$      &2.03  &$1.30\times 10^{-4}$     &2.05          &$3.69\times 10^{-5}$     &2.12\\
\hline
\end{tabular}
\label{jiasaw_error5}
\end{table}

Table \ref{jiasaw_error5} shows the grid refinement analysis. Figure \ref{Fig.lable5} illustrates our results. The designed order of accuracy and convergence is achieved.

\begin{figure}[!ht]
\centering
\subfigure[Numerical solution $u_1$]{
\label{Fig.sub.51}
\includegraphics[width=0.45\textwidth]{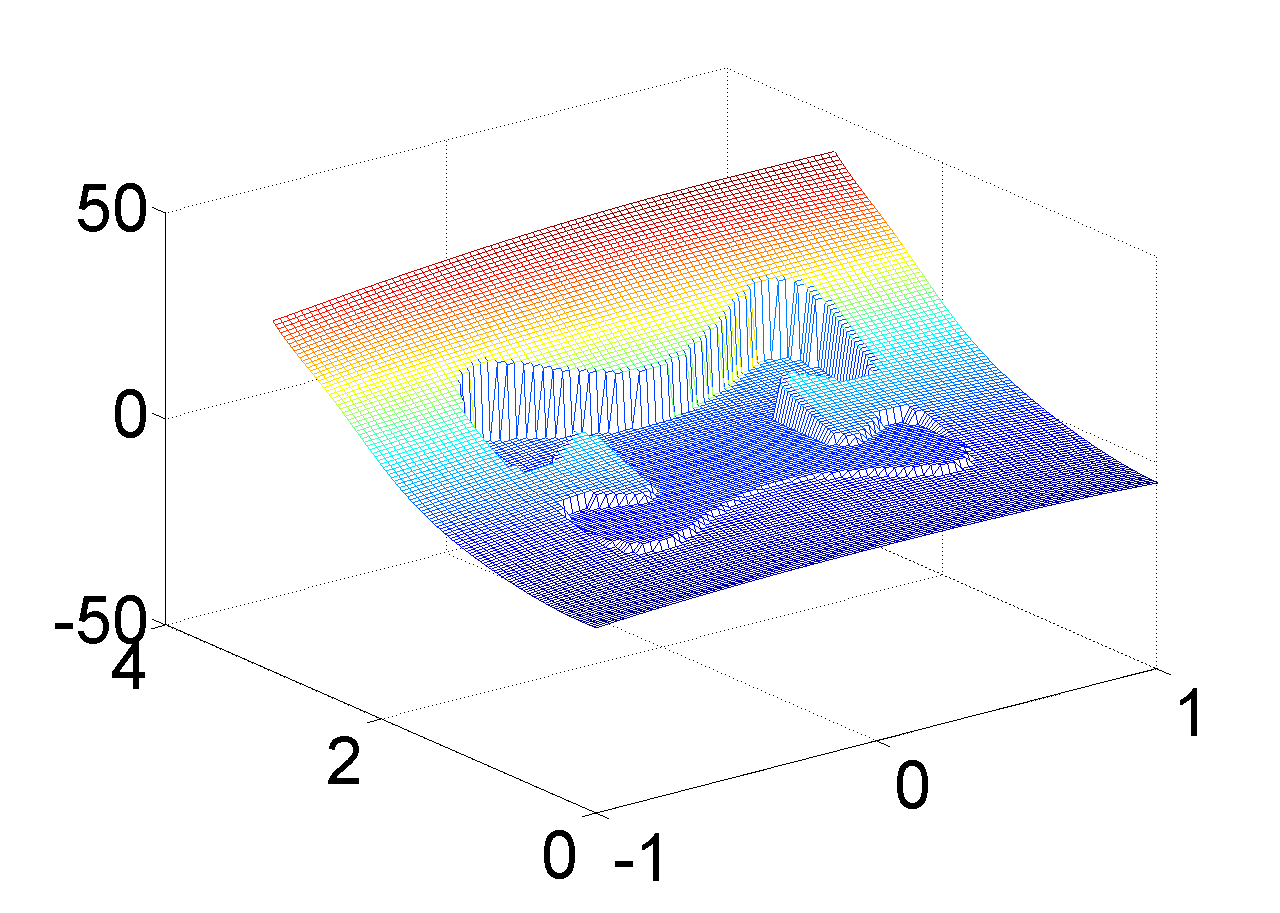}}
\subfigure[Error $u_1$]{
\label{Fig.sub.52}
\includegraphics[width=0.45\textwidth]{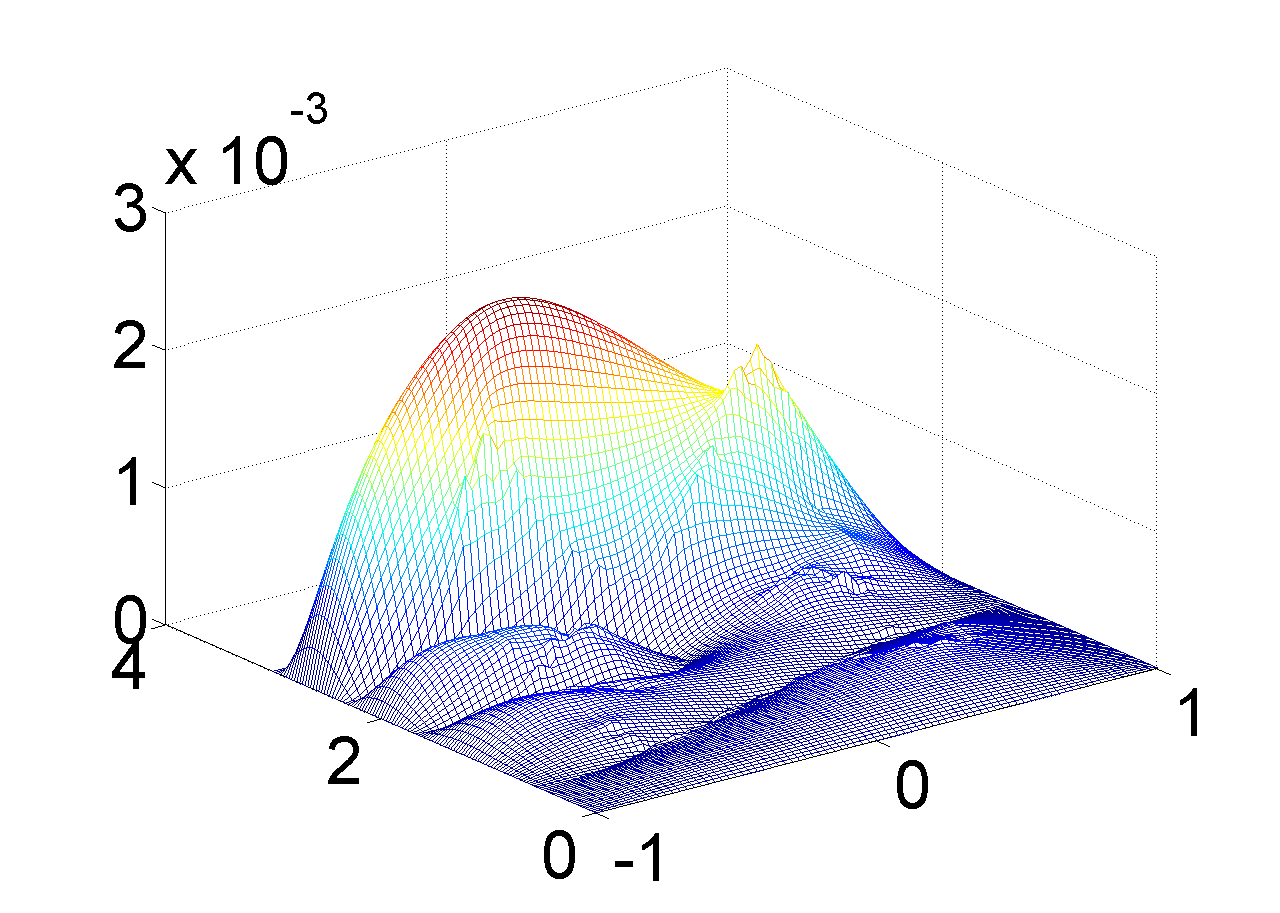}}
\subfigure[Numerical solution $u_2$]{
\label{Fig.sub.53}
\includegraphics[width=0.45\textwidth]{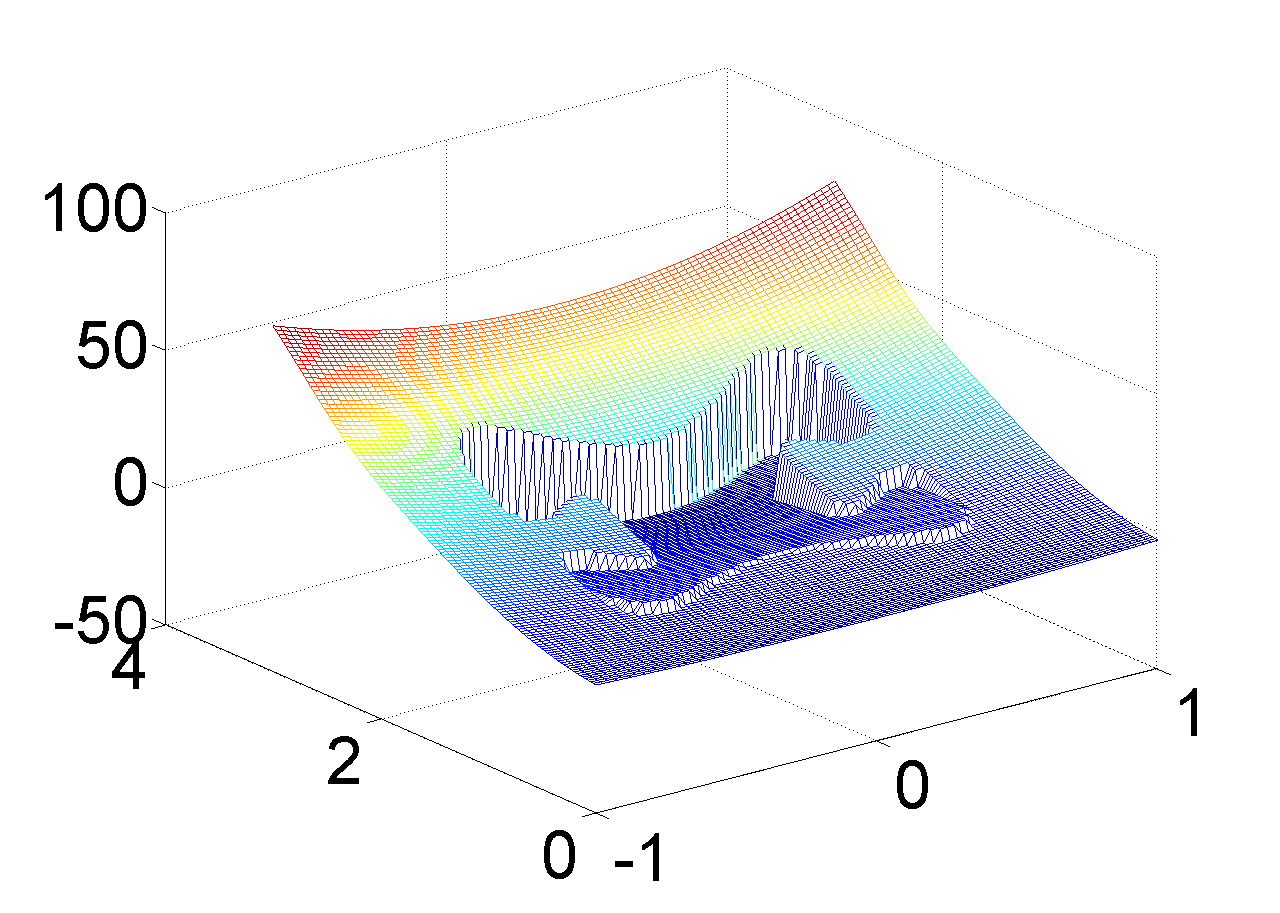}}
\subfigure[Error $u_2$]{
\label{Fig.sub.54}
\includegraphics[width=0.45\textwidth]{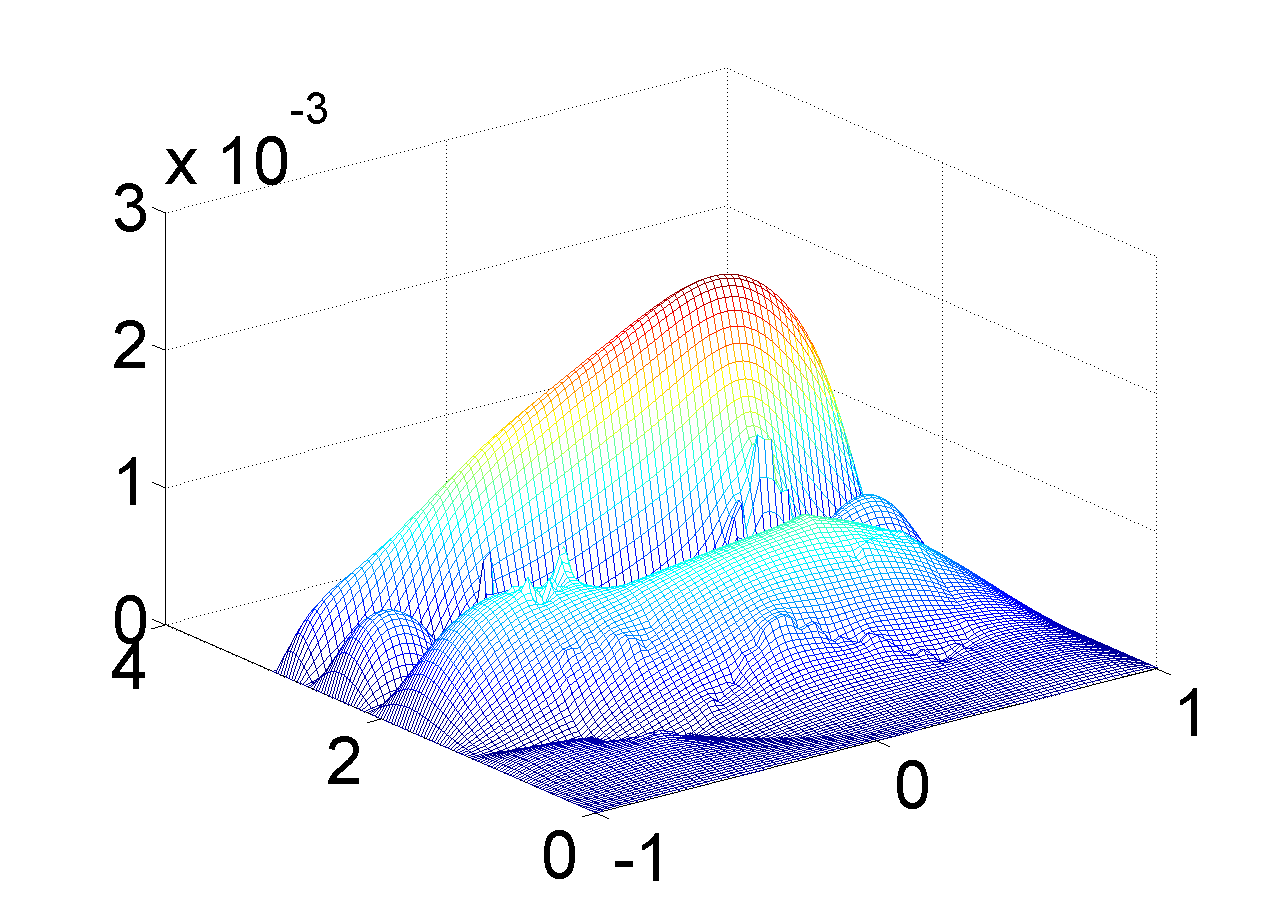}}
\caption{Numerical results for Example 5 on mesh $80\times 120$.}
\label{Fig.lable5}
\end{figure}

\subsection{Inhomogeneous media}
Having validated the MIB method for complex geometry and large contrast in Poisson's ratio and shear modulus, we consider another class of elasticity interface
problems. In many practical elasticity applications, the shear  and bulk moduli may not be constants \cite{Shearer:1999,KLXia:2013d}. Therefore it is important to develop numerical methods for function bulk and shear moduli. In this subsection,  we test our MIB method for handling function bulk and shear moduli.

\paragraph{Example ~6.}
In our first  example, let the domain and interface be the same as those in Example 1.  The exact solution is also designed the same as that  in Example 1.
However, we design the shear modulus to be position dependent function
$$
\mu=
\left\{\begin{array}{ll}
\mu^+=1500000+2000000(x+y), &\ \  \mbox{in}\ \ \Omega^+,\\
\mu^-=2000000+1500000xy, &\ \  \mbox{in}\ \ \Omega^-.
\end{array}\right.
$$
We also set the bulk modulus to
$$
\lambda=
\left\{\begin{array}{ll}
\lambda^+=1000000+4000000(x+y)/3, &\ \  \mbox{in}\ \ \Omega^+,\\
\lambda^-=2000000+1500000xy, &\ \  \mbox{in}\ \ \Omega^-.
\end{array}\right.
$$
The grid refinement analysis is shown in Table \ref{table_var1}. The designed second order accuracy is achieved.

\begin{table}[!ht]
\centering\caption{Numerical error and order for Example 6.}
\begin{tabular}{lllllllll}
\cline{1-9}
$n_x\times n_y$ &$L_\infty(u_1)$&Order &$L_2(u_1)$&Order & $L_\infty(u_2)$&Order &$L_2(u_2)$ & Order  \\
\hline
$20\times 20$ &$5.27\times 10^{-4}$     &           &$2.48\times 10^{-4}$      &   &$2.38\times 10^{-4}$     &                 &$1.17\times 10^{-4}$      &\\
$40\times 40$ &$1.17\times 10^{-4}$     &2.17       &$6.17\times 10^{-5}$      &2.01  &$9.25\times 10^{-5}$     &1.37          &$4.27\times 10^{-5}$    &1.46\\
$80\times 80$ &$2.55\times 10^{-5}$     &2.20       &$1.33\times 10^{-5}$      &2.21  &$2.12\times 10^{-5}$     &2.13          &$9.68\times 10^{-6}$    &2.14\\
$160\times 160$ &$6.23\times 10^{-6}$     &2.03       &$3.26\times 10^{-6}$      &2.03  &$5.06\times 10^{-6}$     &2.07          &$2.38\times 10^{-6}$     &2.03\\
$320\times 320$ &$1.46\times 10^{-6}$     &2.09       &$8.00\times 10^{-7}$      &2.03  &$1.11\times 10^{-6}$     &2.19          &$5.50\times 10^{-7}$     &2.11\\
\hline
\end{tabular}
\label{table_var1}
\end{table}

\paragraph{Example 7.}
In this example, let the domain and interface be the same as those in Example 2. We also adopt the exact solution in Example 2.
We set the shear modulus to be a position dependent function
$$
\mu=
\left\{\begin{array}{ll}
\mu^+=2500000+3000000(x+y), &\ \  \mbox{in}\ \ \Omega^+,\\
\mu^-=3000000+2500000xy, &\ \  \mbox{in}\ \ \Omega^-.
\end{array}\right.
$$
Additionally, we design the following bulk modulus
$$
\lambda=
\left\{\begin{array}{ll}
\lambda^+=5000000+2000000(x+y), &\ \  \mbox{in}\ \ \Omega^+,\\
\lambda^-=3000000+2500000xy, &\ \  \mbox{in}\ \ \Omega^-.
\end{array}\right.
$$

Table \ref{table_var2} presents the grid refinement analysis of this case. Our numerical results are similar to those of Example 2a, which means that the variable shear and bulk moduli do not affect the performance of our method.

\begin{table}[!ht]
\centering
\caption{Numerical error and order for Example 7.}
\begin{tabular}{lllllllll}
\cline{1-9}
$n_x\times n_y$ &$L_\infty(u_1)$&Order &$L_2(u_1)$&Order & $L_\infty(u_2)$&Order &$L_2(u_2)$ & Order  \\
\hline
$20\times 20$ &$4.08\times 10^{-3}$     &           &$1.82\times 10^{-3}$      &      &$1.01\times 10^{-2}$     &              &$3.61\times 10^{-3}$      &\\
$40\times 40$ &$1.01\times 10^{-4}$     &2.01       &$3.25\times 10^{-4}$      &2.49  &$1.97\times 10^{-3}$     &2.36          &$5.54\times 10^{-3}$    &1.45\\
$80\times 80$ &$2.23\times 10^{-4}$     &2.18       &$8.59\times 10^{-5}$      &2.21  &$3.47\times 10^{-4}$     &2.51          &$1.15\times 10^{-4}$    &2.15\\
$160\times 160$ &$5.17\times 10^{-5}$     &2.11       &$2.36\times 10^{-5}$      &1.92  &$8.07\times 10^{-5}$     &2.10          &$2.27\times 10^{-5}$     &2.03\\
$320\times 320$ &$1.40\times 10^{-6}$     &1.88       &$5.26\times 10^{-6}$      &2.17  &$2.81\times 10^{-5}$     &1.52          &$6.63\times 10^{-6}$     &1.78\\
\hline
\end{tabular}
\label{table_var2}
\end{table}

\paragraph{Example 8.}
Finally, we consider another example to validate our method for the combination of complex interface geometry and variable material coefficients.
To this end, we adopt  the domain and interface used  in Example 4.  The exact solution is also designed as that in Example 4. However,
the shear modulus is set to be position dependent
$$
\mu=
\left\{\begin{array}{ll}
\mu^+=2500000+3000000(x+y), &\ \  \mbox{in}\ \ \Omega^+,\\
\mu^-=3000000+2500000xy, &\ \  \mbox{in}\ \ \Omega^-.
\end{array}\right.
$$
The bulk modulus is also a variable function
$$
\lambda=
\left\{\begin{array}{ll}
\lambda^+=5000000+2000000(x+y), &\ \  \mbox{in}\ \ \Omega^+,\\
\lambda^-=3000000+2500000xy, &\ \  \mbox{in}\ \ \Omega^-.
\end{array}\right.
$$

Table \ref{table_var3} gives the grid refinement analysis of this example. We observe the second order accuracy.

\begin{table}[!ht]
\centering
\caption{Numerical error and order for  Example 8.}
\begin{tabular}{lllllllll}
\cline{1-9}
$n_x\times n_y$ &$L_\infty(u_1)$&Order &$L_2(u_1)$&Order & $L_\infty(u_2)$&Order &$L_2(u_2)$ & Order  \\
\hline
$40\times 40$ &$1.69\times 10^{-3}$     &           &$7.18\times 10^{-4}$      &   &$1.80\times 10^{-3}$     &                 &$8.20\times 10^{-4}$      &\\
$80\times 80$ &$3.94\times 10^{-4}$     &2.28       &$1.88\times 10^{-4}$      &1.93  &$3.09\times 10^{-4}$     &2.54          &$1.50\times 10^{-4}$    &2.45\\
$160\times 160$ &$8.45\times 10^{-5}$     &2.22       &$3.58\times 10^{-5}$      &2.39  &$7.13\times 10^{-5}$     &2.16          &$3.49\times 10^{-5}$    &2.09\\
$320\times 320$ &$1.87\times 10^{-5}$     &2.18       &$1.03\times 10^{-5}$      &1.80  &$1.89\times 10^{-5}$     &1.92          &$1.01\times 10^{-5}$     &1.79\\
\hline
\end{tabular}
\label{table_var3}
\end{table}

\section{Conclusion}

In this work,  elasticity interface problems governed by the linear elasticity theory is investigated by the matched interface and boundary (MIB) method for the first time.  Both isotropic homogeneous material and  isotropic inhomogeneous material are considered in the theoretical modeling. In particular, the isotropic inhomogeneous material  is described by a strain-stress constitutive law with a function-type of modulus. We analyze both strongly discontinuous and weakly discontinuous  solutions to the governing elasticity equations. A new MIB scheme is developed for this class of elasticity interface problems.

Unlike elliptic interface problems, the elasticity interface problems involve more governing equations and more complicated  interface jump conditions. As such,  we need to extend the original MIB method designed elliptic interface problems \cite{Yu:2007c,Yu:2007a,Zhao:2004,Zhou:2006c,Zhou:2006d} for multiple equations and interface conditions. Additionally, the MIB method developed in the present study has to take a particular care for cross derivatives in the governing elasticity equations. Such cross derivatives can be very difficult to deal with when the interface geometry is very complex.    To account for these terms and local interface geometry, we have modified the common fictitious definition, designed the new fictitious schemes, and made use of the secondary fictitious values.

Numerous analytical tests are designed to examine the accuracy, investigate the convergence and explore the robustness of the present MIB method.  Four types of complex interface geometries, namely,  circle, ellipse, flower-liked and jigsaw-liked interfaces, are employed in our study. Additionally, both weakly discontinuous and strongly discontinuous solutions are considered in our tests. Moreover, we have also examined our method for both small and large contrasts in Poisson's ratios and shear moduli.   Finally, both constant  material parameters and variable material parameters are utilized in our validation.
An essentially second-order  accuracy in both $L_\infty$ and $L_2$  norms is observed  in all tests.

\section*{Acknowledgments}

This work was supported in part by NSF grants   IIS-1302285 and DMS-1160352,   NIH grant R01GM-090208 and MSU Center for Mathematical Molecular Biosciences Initiative.

\section*{Literature cited}
\renewcommand\refname{}

\bibliographystyle{abbrv}
\bibliography{refs}

\begin{thebibliography}{10}

\bibitem{Anandarajah:2010}
A.~Anandarajah.
\newblock {\em Computational Methods in Elasticity and Plasticity:Solids and
  Porous Media}.
\newblock Springer, 2010.

\bibitem{Becker:2009}
R.~Becker, E.~Burman, and P.~Hansbo.
\newblock A nitsche extended finite element method for incompressible
  elasticity with discontinuous modulus of elasticity.
\newblock {\em Comput. Methods Appl. Mech. Engrg.}, 198:3352--3360, 2009.

\bibitem{YChang:2012}
Y.~Z. Chang.
\newblock Adaptive finite element method for elasticity interface problems.
\newblock {\em Journal of Computational Mathematics}, 30:629 – 642, 2012.

\bibitem{DuanChen:2011a}
D.~Chen, Z.~Chen, C.~Chen, W.~H. Geng, and G.~W. Wei.
\newblock {MIBPB}: A software package for electrostatic analysis.
\newblock {\em J. Comput. Chem.}, 32:657 -- 670, 2011.

\bibitem{Dvorak:2013}
G.~Dvorak.
\newblock {\em Micromechanics of composite materials}.
\newblock Springer, 2013.

\bibitem{Eshelby:1956}
J.~D. Eshelby.
\newblock The continuum theory of lattice defects.
\newblock {\em Progress in solid state physics}, 3:79--144, 1956.

\bibitem{Eshelby:1957}
J.~D. Eshelby.
\newblock The determination of the elastic field of an ellipsoidal inclusion,
  and related problems.
\newblock {\em Proceedings of the Royal Society London A,}, 241:376--396, 1957.

\bibitem{fonberg:1998}
B.~Fornberg.
\newblock Calculation of weights in finite difference formulas.
\newblock {\em SIAM Rev}, 40:685--691, 1998.

\bibitem{Fries:2010}
T.~P. Fries and T.~Belyschko.
\newblock The extended/generalized finite element method: An overview of the
  method and its applications.
\newblock {\em Int. J. Numer. Meth. Engng.}, 84:253--304, 2010.

\bibitem{Geng:2011}
W.~Geng and G.~W. Wei.
\newblock Multiscale molecular dynamics using the matched interface and
  boundary method.
\newblock {\em J Comput. Phys.}, 230(2):435--457, 2011.

\bibitem{Geng:2007a}
W.~Geng, S.~Yu, and G.~W. Wei.
\newblock Treatment of charge singularities in implicit solvent models.
\newblock {\em Journal of Chemical Physics}, 127:114106, 2007.

\bibitem{YGong:2010}
Y.~Gong and Z.~L. Li.
\newblock Immersed interface finite element methods for elasticity interface
  problems with non-homogeneous jump conditions.
\newblock {\em Numerical Mathematics -- Theoretical Methods and Applications},
  3:23 – 39, 2010.

\bibitem{Hansbo:2002}
A.~Hansbo and P.~Hansbo.
\newblock An unfitted finite element method.
\newblock {\em Comput. Methods Appl. Mech. Engng}, 191:5537--5552, 2002.

\bibitem{HouSM:2012}
S.~Hou, Z.~Li, L.~Wang, and W.~Wang.
\newblock A numerical method for solving elasticity equations with sharp-edged
  interfaces.
\newblock {\em Commun. Comput. Phys.}, 12:595--612, 2012.

\bibitem{LeVeque:1994}
R.~J. LeVeque and Z.~L. Li.
\newblock The immersed interface method for elliptic equations with
  discontinuous coefficients and singular sources.
\newblock {\em SIAM J. Numer. Anal.}, 31:1019--1044, 1994.

\bibitem{LiZL:2005}
Z.~L. Li and X.~Z. Yang.
\newblock An immersed fem for elasticity equations with interfaces.
\newblock {\em AMS Contemporary Mathematics}, 383:285--298, 2005.

\bibitem{LinT:2013}
T.~Lin, D.~W. Sheen, and X.~Zhang.
\newblock A locking-free immersed finite element method for planar elasticity
  interface problems.
\newblock {\em Journal of Comput. Phys.}, 247:228--247, 2013.

\bibitem{LinT:2012}
T.~Lin and X.~Zhang.
\newblock Linear and bilinear immersed finite elements for planar elasticity
  interface problems.
\newblock {\em Journal of Computational and Applied Mathematics},
  236:4681--4699, 2012.

\bibitem{Mathiesen:2008}
J.~Mathiesen, I.~Procaccia, and I.~Regev.
\newblock Elasticity with arbitrarily shaped inhomogeneity.
\newblock {\em Physical Review E,}, 77:026606, 2008.

\bibitem{Mergheim:2006}
J.~Mergheim.
\newblock {\em Computational Modeling of Strong and Weak Discontinuities}.
\newblock PhD thesis, Technical University of Kaiserslautern, 2006.

\bibitem{Michaeli:2013}
M.~Michaeli, F.~Assous, and A.~Golubchik.
\newblock A nitsche type method for stress fields calculation in dissimilar
  material with interface crack.
\newblock {\em Applied Numerical Mathematics}, 67:187--203, 2013.

\bibitem{Shearer:1999}
P.~M. Shearer.
\newblock {\em Introduction to seismology}.
\newblock Cambridge University Press, 1999.

\bibitem{Stolarska:2001}
M.~Stolarska, D.~L. Chopp, N.~Moes, and T.~Belytschko.
\newblock Modelling crack growth by level sets in the extended finite element
  method.
\newblock {\em Int. J. Numer. Meth. Engng}, 51:943--960, 2001.

\bibitem{Sukumar:2001}
N.~Sukumar, D.~L. Chopp, N.~Moes, and T.~Belytschko.
\newblock Modeling holes and inclusions by level sets in the extended
  finite-element method.
\newblock {\em Comput. Methods Appl. Mech. Engrg.}, 190:6180--6200, 2001.

\bibitem{Theillard:2013}
M.~Theillard, L.~F. Djodom, J.~L. Vie, and F.~Gibou.
\newblock A second-order sharp numerical method for solving the linear
  elasticity equations on irregular domains and adaptive grids - application to
  shape optimization.
\newblock {\em Journal of Computational Physics}, 233:430--448, 2013.

\bibitem{WangXS:2009}
X.~S. Wang, L.~T. Zhang, and W.~K. Liu.
\newblock On computational issues of immersed ﬁnite element methods.
\newblock {\em Journal of Comput. Phys.}, 228:2535--2551, 2009.

\bibitem{JCPWei:1999}
G.~W. Wei.
\newblock Discrete singular convolution for the solution of the {Fokker-Planck}
  equations.
\newblock {\em J. Chem. Phys.}, 110:8930-- 8942, 1999.

\bibitem{Wei:2009}
G.~W. Wei.
\newblock Differential geometry based multiscale models.
\newblock {\em Bulletin of Mathematical Biology}, 72:1562 -- 1622, 2010.

\bibitem{Wei:2013}
G.-W. Wei.
\newblock Multiscale, multiphysics and multidomain models {I: Basic} theory.
\newblock {\em Journal of Theoretical and Computational Chemistry},
  12(8):1341006, 2013.

\bibitem{CTWu:2013}
C.~T. Wu, Y.~Guo, and E.~Askari.
\newblock Numerical modeling of composite solids using an immersed meshfree
  galerkin method.
\newblock {\em Composites Part B: Engineering}, 45:1397 – 1413, 2013.

\bibitem{KLXia:2013d}
K.~L. Xia, K.~Opron, and G.~W. Wei.
\newblock Multiscale multiphysics and multidomain models --- { Flexibility} and
  rigidity.
\newblock {\em Journal of Chemical Physics}, 139:194109, 2013.

\bibitem{KLXia:2011}
K.~L. Xia, M.~Zhan, and G.-W. Wei.
\newblock The matched interface and boundary {(MIB)} method for multi-domain
  elliptic interface problems.
\newblock {\em Journal of Computational Physics}, 230:8231--8258, 2011.

\bibitem{XieH:2011}
H.~Xie, Z.~L. Li, and Z.~H. Qiao.
\newblock A finite element method for elasticity interface problems with
  locally modified triangulations.
\newblock {\em International Journal of Numerical Analysis and Modeling},
  8:189--200, 2011.

\bibitem{XuZL:2003}
Z.~L. Xu, J.~Glimm, and X.~L. Li.
\newblock Front tracking algorithm using adaptively refined meshes.
\newblock In {\em Adaptive Mesh Refinement - Theory and Applications, the
  Lecture Notes in Computational Science and Engineering}, pages 83--89,
  September 2003.

\bibitem{YangXZ:2003}
X.~Z. Yang, B.~Li, and Z.~L. Li.
\newblock The immersed interface method for elasticity problems with interface.
\newblock {\em Dyn. Contin. Discrete Impuls. Syst. Ser. A Math. Anal.},
  10:783--808, 2003.

\bibitem{Yu:2007}
S.~N. Yu, W.~H. Geng, and G.~W. Wei.
\newblock Treatment of geometric singularities in implicit solvent models.
\newblock {\em Journal of Chemical Physics}, 126:244108, 2007.

\bibitem{Yu:2007a}
S.~N. Yu and G.~W. Wei.
\newblock Three-dimensional matched interface and boundary {(MIB)} method for
  treating geometric singularities.
\newblock {\em J. Comput. Phys.}, 227:602--632, 2007.

\bibitem{Yu:2007c}
S.~N. Yu, Y.~C. Zhou, and G.~W. Wei.
\newblock Matched interface and boundary ({MIB}) method for elliptic problems
  with sharp-edged interfaces.
\newblock {\em J. Comput. Phys.}, 224(2):729--756, 2007.

\bibitem{SZhao:2008a}
S.~Zhao.
\newblock Full-vectorial matched interface and boundary ({MIB}) method for the
  modal analysis of dielectric waveguides.
\newblock {\em IEEE/OSA Journal of Lighwave Technology}, 26:2251--2259, 2008.

\bibitem{SZhao:2010a}
S.~Zhao.
\newblock High order matched interface and boundary methods for the {Helmholtz}
  equation in media with arbitrarily curved interfaces.
\newblock {\em J. Comput. Phys.}, 229:3155--3170, 2010.

\bibitem{Zhao:2004}
S.~Zhao and G.~W. Wei.
\newblock High-order {FDTD} methods via derivative matching for {Maxwell's}
  equations with material interfaces.
\newblock {\em J. Comput. Phys.}, 200(1):60--103, 2004.

\bibitem{QZheng:2011a}
Q.~Zheng, D.~Chen, and G.~W. Wei.
\newblock Second-order {Poisson-Nernst-Planck} solver for ion transport.
\newblock {\em Journal of Comput. Phys.}, 230:5239 -- 5262, 2011.

\bibitem{QZheng:2011b}
Q.~Zheng and G.~W. Wei.
\newblock {Poisson-Boltzmann-Nernst-Planck model}.
\newblock {\em Journal of Chemical Physics}, 134:194101, 2011.

\bibitem{Zhou:2008b}
Y.~C. Zhou, M.~Feig, and G.~W. Wei.
\newblock Highly accurate biomolecular electrostatics in continuum dielectric
  environments.
\newblock {\em Journal of Computational Chemistry}, 29:87--97, 2008.

\bibitem{YCZhou:2012a}
Y.~C. Zhou, J.~G. Liu, and D.~L. Harry.
\newblock A matched interface and boundary method for solving multi-flow
  navier-stokes equations with applications to geodynamics.
\newblock {\em Journal of Computational Physics}, 231:223--242, 2012.

\bibitem{Zhou:2006d}
Y.~C. Zhou and G.~W. Wei.
\newblock On the fictitious-domain and interpolation formulations of the
  matched interface and boundary ({MIB}) method.
\newblock {\em J. Comput. Phys.}, 219(1):228--246, 2006.

\bibitem{Zhou:2006c}
Y.~C. Zhou, S.~Zhao, M.~Feig, and G.~W. Wei.
\newblock High order matched interface and boundary method for elliptic
  equations with discontinuous coefficients and singular sources.
\newblock {\em J. Comput. Phys.}, 213(1):1--30, 2006.

\end{thebibliography}

\end{document}